\documentclass[11pt]{article}

\usepackage{titlesec,clipboard}
\usepackage[margin=1in]{geometry}

\titleformat{\section}
  {\normalfont\fontsize{12}{12}\bfseries}{\thesection}{1em}{}
\titleformat{\subsection}
  {\normalfont\fontsize{10}{10}\bfseries}{\thesubsection}{1em}{}
\usepackage{amsmath,amssymb,amsthm,tikz,youngtab,multicol,xcolor,setspace}
\usetikzlibrary{arrows,decorations.pathreplacing,decorations.markings,patterns}
\usepackage{empheq,bbm}
\usepackage{subcaption,float,authblk}

\DeclareMathOperator{\fpdim}{FPdim}

\newtheoremstyle{mytheoremstyle}{1em}{0pt}{\itshape}{}{\bf}{.}{.5em}{} 
\newtheoremstyle{mytheoremstyle2}{1em}{0pt}{}{}{\bf}{.}{.5em}{} 

\theoremstyle{mytheoremstyle}
\newtheorem{theo}{Proposition}[section]
\newtheorem{lem}[theo]{Lemma}
\newtheorem{cor}[theo]{Corollary}

\newtheorem{innercustomthm}{Proposition}
\newenvironment{customthm}[1]
  {\renewcommand\theinnercustomthm{#1}\innercustomthm}
  {\endinnercustomthm}

\newtheorem{innercustomlem}{Lemma}
\newenvironment{customlem}[1]
  {\renewcommand\theinnercustomlem{#1}\innercustomlem}
  {\endinnercustomlem}

\newtheorem{innercustomcor}{Corollary}
\newenvironment{customcor}[1]
  {\renewcommand\theinnercustomcor{#1}\innercustomcor}
  {\endinnercustomcor}

\theoremstyle{theorem}
\newtheorem*{theorem}{Theorem}
\theoremstyle{mytheoremstyle2}
\newtheorem{note}[theo]{Note}
\newtheorem{defi}[theo]{Definition}
\newtheorem{example}[theo]{Example}
\newtheorem{question}[theo]{Question}
\newtheorem{conjecture}[theo]{Conjecture}

\title{\Large Algebraic number fields generated by \\ Frobenius-Perron dimensions in fusion rings}
\date{}
\author{\large Terry Gannon and Andrew Schopieray}

\begin{document}
\parskip = \baselineskip
\setlength{\parindent}{0cm}

\maketitle

\begin{abstract}
\noindent From a unifying lemma concerning fusion rings, we prove a collection of number-theoretic results about fusion, braided, and modular tensor categories.  First, we prove that every fusion ring has a dimensional grading by an elementary abelian 2-group.  As a result, we bound the order of the multiplicative central charge of arbitrary modular tensor categories.  We also introduce Galois-invariant subgroups of the Witt group of nondegenerately braided fusion categories corresponding to algebraic number fields generated by Frobenius-Perron dimensions.  Lastly, we provide a complete description of the fields generated by the Frobenius-Perron dimensions of simple objects in $\mathcal{C}(\mathfrak{g},k)$, the modular tensor categories arising from the representation theory of quantum groups at roots of unity, as well as the fields generated by their Verlinde eigenvalues.    
\end{abstract}


\begin{section}*{Introduction}

The study of fusion rings is primarily motivated by those which can be realized as Grothendieck rings of fusion categories, a structured class of semisimple monoidal categories with finitely many isomorphism classes of simple objects.  Meanwhile, the study of fusion categories and their relatives (braided, ribbon, and modular tensor categories) has found application to topological quantum computing \cite{ericwang}, invariants of knots and 3-manifolds \cite{turaev}, subfactor theory \cite{mug2,mug3}, vertex operator algebras \cite{huang2005vertex} and conformal field theory \cite{kawaconformal}, to name a few.  As weakly-defined objects, it is perhaps not surprising that results about arbitrary fusion rings are few and far between with \cite{tcat,ENO} being common references.  Our guiding observation is an elementary, but novel, fact about Frobenius-Perron eigenvalues of fusion matrices under additive decomposition.

\begin{customlem}{\ref{theorem}}
Let $R$ be a fusion ring.  If $x_1,x_2\in R$ are $\mathbb{Z}_{\geq0}$-linear combinations of basis elements of $R$, then $\mathrm{FPdim}(x_1),\mathrm{FPdim}(x_2)\in\mathbb{Q}(\mathrm{FPdim}(x_1+x_2))$.
\end{customlem}

\par Though elementary, Lemma \ref{theorem} allows for statements about weakly integral fusion rings (i.e. \!$\mathrm{FPdim}(R)\in\mathbb{Z}$) to be extended to arbitrary number fields.  We provide four examples of this philosophy but it is clear that many more exist.  The following notation will be used extensively throughout the paper.  For a fusion ring $R$, define $\mathbb{K}_x:=\mathbb{Q}(\mathrm{FPdim}(x))$ for any $x\in R$, $\mathbb{K}_0:=\mathbb{Q}(\mathrm{FPdim}(R))$, and $\mathbb{K}_1:=\mathbb{Q}(\mathrm{FPdim}(x):x\in R)$.  If $\mathcal{C}$ is a fusion category, we will use the same notation where $R$ is the Grothendieck ring of $\mathcal{C}$.  We include a table of recurring notation in Figure \ref{fig:X}.

\begin{paragraph}{Application 1.} Gelaki and Nikshych \cite[Theorem 3.10]{nilgelaki} identify a faithful grading of weakly integral fusion rings, which has been useful in determining the structure of solvable, nilpotent, and weakly group-theoretical fusion categories.  Refining Lemma 1.1, we can generalize their grading to any fusion ring.   Consider for instance any product $X$ of basis elements, and let $b_i$ be any basis element appearing in $X$. Then $\mathbb{Q}(\mathrm{FPdim}(b_i))=\mathbb{Q}(\mathrm{FPdim}(X))$ (Lemma \ref{theorem}  only gives containment). In other words, we have the following proposition.
\begin{customthm}{\ref{intrograd}}
If $R$ is a fusion ring, $R$ is faithfully graded by $\mathrm{Gal}(\mathbb{K}_1/\mathbb{K}_0)$.
\end{customthm}
This result could also be gleaned from a result on table algebras due to Blau \cite[Theorem 7.5]{MR3108081}.  The group $\mathrm{Gal}(\mathbb{K}_1/\mathbb{K}_0)$ is an elementary abelian 2-group (possibly trivial) and the trivial component of this grading is the fusion subring consisting of all $x\in R$ with $\mathrm{FPdim}(x)\in\mathbb{K}_0$.  By \cite[Corollary 3.7]{nilgelaki}, $\mathrm{Gal}(\mathbb{K}_1/\mathbb{K}_0)$ is isomorphic to a quotient of the universal grading group $U(R)$.  Thus even modest restrictions on the universal grading group of a fusion ring has powerful repurcussions for the potential Frobenius-Perron dimensions of elements.
\end{paragraph}

\begin{paragraph}{Application 2.}  We prove that the order of the multiplicative central charge $\xi(\mathcal{C})$ of a modular tensor category $\mathcal{C}$ (over $\mathbb{C}$) is constrained by the exponent of $\mathrm{Gal}(\mathbb{L}_\mathbbm{1}/\mathbb{Q})$, where
\begin{equation}
\mathbb{L}_\mathbbm{1}:=\mathbb{Q}(\dim(X):\text{simple }X\in\mathcal{C}).
\end{equation}
This fact first appeared 20 years ago in the special case $\mathbb{L}_\mathbbm{1}=\mathbb{Q}$ in an unpublished work of Coste and the first author \cite[Proposition 3(b)]{gannonunpub}.  This special case was later revisited by Dong, Lin and Ng \cite[Theorem 6.10]{dong2015congruence}.  We generalize these arguments to any algebraic number field.  For all $n\in\mathbb{Z}_{\geq1}$ define $f(n)$ to be the largest $m\in\mathbb{Z}_{\geq2}$ such that the exponent of $(\mathbb{Z}/m\mathbb{Z})^\times$ is equal to $n$.

\begin{customcor}{\ref{bettercor}}
Let $\mathcal{C}$ be a modular tensor category and let $N\in\mathbb{Z}_{\geq1}$ be the exponent of $\mathrm{Gal}(\mathbb{L}_\mathbbm{1}/\mathbb{Q})$.  Then $\xi(\mathcal{C})$ is a root of unity of order dividing $f(2N)/3$.
\end{customcor}
When $\mathcal{C}$ is pseudounitary, $\mathbb{L}_\mathbbm{1}=\mathbb{K}_1$ and so Proposition \ref{intrograd} with Corollary \ref{bettercor} imply that $\xi(\mathcal{C})$ is a root of unity of order dividing $2f(2N)/3$, where $N$ is the exponent of $\mathrm{Gal}(\mathbb{K}_0/\mathbb{Q})$.
\end{paragraph}

\begin{paragraph}{Application 3.}  The Witt group $\mathcal{W}$ of nondegenerately braided fusion categories, introduced by Davydov, M\"uger, Nikshych, and Ostrik \cite{DMNO}, consists of equivalence classes $[\mathcal{C}]$ of nondegenerately braided fusion categories $\mathcal{C}$ modulo Drinfeld centers, with abelian group operation $[\mathcal{C}][\mathcal{D}]:=[\mathcal{C}\boxtimes\mathcal{D}]$.  We define a new invariant of Witt classes $\Phi([\mathcal{C}])$ as the intersection of $\mathbb{K}_0$ over all nondegenerately braided fusion categories Witt equivalent to $\mathcal{C}$.  Multiplicative central charge is an invariant of Witt equivalence \cite[Lemma 5.27]{DMNO} and this has been partially generalized to higher multiplicative central charges as well \cite[Theorem 6.1]{schopieraywang}.  But multiplicative central charge is only well-defined for (pre-)modular tensor categories, while $\Phi([\mathcal{C}])$ is defined on all of $\mathcal{W}$.  Of most interest is the following consequence.
\begin{customthm}{\ref{wittt}}
Let $\mathbb{K}$ be an algebraic number field.  The set $\mathcal{W}_\mathbb{K}:=\{[\mathcal{C}]\in\mathcal{W}:\Phi([\mathcal{C}])\subset\mathbb{K}\}$
is a Galois-invariant subgroup of $\mathcal{W}$.
\end{customthm}
One can think of $\mathcal{W}_\mathbb{K}$ as a generalization of $\mathcal{W}_\mathrm{pt}\subset\mathcal{W}$, the subgroup of equivalence classes of pointed modular tensor categories, whose structure is explicitly known \cite[Section 5.3]{DMNO}.
\end{paragraph}

\begin{paragraph}{Application 4.}  We prove $\mathbb{K}_\lambda=\mathbb{Q},\mathbb{K}_0,\mathbb{K}_1$ for all simple $\lambda\in\mathcal{C}(\mathfrak{g},k)$, the modular tensor categories arising from the representation theory of quantum groups at roots of unity, using (i) a classification of fusion subcategories of $\mathcal{C}(\mathfrak{g},k)$ \cite[Theorem 1]{Sawin06} and (ii) inspection of the $\fpdim$ of at most two objects in the category.  We explicitly describe these fields in Figure \ref{fig:B} and tabulate the categories for which $[\mathbb{K}_0:\mathbb{Q}]\leq9$ in Figure \ref{fig:A}.  Most surprisingly, all fields involved are of the form $\mathbb{Q}_n:=\mathbb{Q}(\cos(2\pi/n))$ or $\mathbb{Q}(\sqrt{n})$ for some $n\in\mathbb{Z}_{\geq1}$ with precisely 2 exceptions: $\mathcal{C}(F_4,4)$ and $\mathcal{C}(E_8,5)$.  The exception $\mathbb{Q}(\mathrm{FPdim}(\mathcal{C}(F_4,4)))$ is equal to $\mathbb{Q}(\mathrm{FPdim}(\mathcal{Z}(\mathcal{H})))$ where $\mathcal{Z}(\mathcal{H})$ is the Drinfeld center of the fusion category corresponding to the extended Haagerup subfactor \cite[Equation(s) 9]{exthaag}.  The algebraic number fields generated by FPdims must be real and cyclotomic \cite[Corollary 8.54]{DNO}, i.e. \!subfields of $\mathbb{Q}_n$, but by any (figurative) measure, the fields $\mathbb{K}_\lambda$ represent a measure zero set of all real cyclotomic fields.

\par To illustrate the sparsity of $\mathbb{Q}_n$, consider the real cyclotomic fields $\mathbb{Q}(\sqrt{m_1},\ldots,\sqrt{m_k})$ where $m_j$ are distinct square-free integers greater than $1$ for all $1\leq j\leq k$.   Of these, the only of the form $\mathbb{Q}_n$ are $\mathbb{Q}_5=\mathbb{Q}_{10}=\mathbb{Q}(\sqrt{5})$, $\mathbb{Q}_8=\mathbb{Q}(\sqrt{2})$, $\mathbb{Q}_{12}=\mathbb{Q}(\sqrt{3})$, and $\mathbb{Q}_{24}=\mathbb{Q}(\sqrt{2},\sqrt{3})$.  For another measure of the sparsity of $\mathbb{Q}_n$, note that $[\mathbb{Q}_n:\mathbb{Q}]=\varphi(n)/2$ where $\varphi(n)$ is the Euler totient function.  It is well-known that the Euler totient function is not surjective on the even positive integers.  Hence there does not exist $\mathbb{Q}_n$ of degree $m$ over $\mathbb{Q}$ for any $m\in\mathbb{Z}_{\geq1}$ which is not in the image of $\varphi/2$.  E.g. \!there does not exist $\mathbb{K}_\lambda$ for simple $\lambda\in\mathcal{C}(\mathfrak{g},k)$ such that
\begin{equation}\label{inf}
[\mathbb{K}_\lambda:\mathbb{Q}]=7,13,17,19,25,31,34,37,38,43,45,47,49,57,59,61,62,134, 71, 73, 76,\ldots
\end{equation}
List (\ref{inf}) is infinite (e.g. \!\cite[Theorem 3.1]{coleman2012remarks}).  Finally we use knowledge of $\mathbb{K}_1$ to explicitly describe the fields $\mathbb{L}$ generated by all Verlinde eigenvalues for $\mathcal{C}(\mathfrak{g},k)$.  Set $\xi_n:=\exp(2\pi i/n)$ for $n\in\mathbb{Z}_{\geq1}$.  If $N\in\mathbb{Z}_{\geq1}$ is the conductor of a modular tensor category $\mathcal{C}$, the modular data of $\mathcal{C}$ is contained in $\mathbb{Q}(\xi_N)$ and \cite[Proposition 6.7(iii)]{dong2015congruence} implies $\mathrm{Gal}(\mathbb{Q}(\xi_N)/\mathbb{L})$ is an elementary abelian 2-group.
\end{paragraph}
%

\begin{figure}[H]
\centering
\begin{equation*}
\begin{array}{|rc|c|c|c|c|}
\hline \mathrm{Type}& \mathrm{Conditions} & \mathbb{K}_0 & \mathbb{K}_1 &  \mathrm{Exceptions}\\\hline\hline
A_{n,k}:&\text{ odd }n\text{ and even }\kappa & \mathbb{Q}_\kappa & \mathbb{Q}_{2\kappa} & \\
&\text{else}& \mathbb{Q}_\kappa & \mathbb{Q}_{\kappa} &k=1\\\hline
B_{n,k}:&\text{odd }n & \mathbb{Q}_{2\kappa} & \mathbb{Q}_{4\kappa} & k=1,2 \\
&\text{even }n & \mathbb{Q}_{2\kappa} & \mathbb{Q}_{2\kappa} & k=1,2 \\\hline
C_{n,k}:& & \mathbb{Q}_{2\kappa} & \mathbb{Q}_{2\kappa} & k=1 \\\hline
D_{n,k}:&n\equiv2,3\,\,(\mathrm{mod}\,\,4)\text{ and even }\kappa & \mathbb{Q}_\kappa  & \mathbb{Q}_{2\kappa} & k=2  \\
&\text{else}& \mathbb{Q}_\kappa  & \mathbb{Q}_\kappa &  k=1,2 \\\hline
E_{6,k}: & &\mathbb{Q}_{\kappa} & \mathbb{Q}_{\kappa} & k=1,3 \\\hline
E_{7,k}: & \text{even }\kappa&\mathbb{Q}_{\kappa} & \mathbb{Q}_{2\kappa} & k=2 \\
 &\text{odd }\kappa &\mathbb{Q}_{\kappa} & \mathbb{Q}_{\kappa} & k=1,3 \\\hline
E_{8,k}:&  &\mathbb{Q}_{\kappa} & \mathbb{Q}_{\kappa} & k=1,2,3,5 \\\hline
F_{4,k}:& & \mathbb{Q}_{2\kappa} & \mathbb{Q}_{2\kappa} & k=1,3,4 \\\hline
G_{2,k}:& & \mathbb{Q}_{3\kappa} & \mathbb{Q}_{3\kappa} & k=1,3 \\\hline
 \end{array}
\end{equation*}
    \caption{Frobenius-Perron dimension fields for $\mathcal{C}(\mathfrak{g},k)$}%
    \label{fig:B}%
\end{figure}

\begin{figure}[H]
\centering
\begin{equation*}
\begin{array}{|c|c|}
\hline \mathbb{K}_0 & \text{Type}\\\hline\hline
\mathbb{Q}=\mathbb{Q}_3=\mathbb{Q}_4=\mathbb{Q}_6 & \mathbf{A_{n,1}},A_{n,5-n},\mathbf{B_{n,1}},\mathbf{B_{n,2}},C_{4,1},\mathbf{D_{n,1}},\mathbf{D_{n,2}},\mathbf{E_{6,1}},\mathbf{E_{7,1}},\mathbf{E_{8,1}} \\\hline\hline\hline
\mathbb{Q}_5=\mathbb{Q}_{10} & A_{n,9-n},C_{n,6-n},D_{4,4},\mathbf{E_{6,3}},\mathbf{E_{7,2}},\mathbf{F_{4,1}},\mathbf{G_{2,1}} \\\hline
\mathbb{Q}_8 & A_{n,7-n},C_{2,3},C_{3,2},\mathbf{E_{7,2}},\mathbf{E_{8,2}}  \\\hline
\mathbb{Q}_{12} & A_{n,11-n},C_{n,7-n},D_{4,6},D_{5,4} \\\hline
\mathbb{Q}(\sqrt{6}) & \mathbf{F_{4,3}} \\\hline
\mathbb{Q}(\sqrt{21}) & \mathbf{E_{7,3}},\mathbf{G_{2,3}} \\\hline\hline\hline
\mathbb{Q}_7=\mathbb{Q}_{14} & A_{n,6-n},A_{n,13-n},C_{n,8-n},D_{n,16-2n},E_{6,2} \\\hline
\mathbb{Q}_9=\mathbb{Q}_{18} & A_{n,8-n},A_{n,17-n},B_{3,4},C_{n,10-n},D_{4,3},D_{n,20-2n},G_{2,2},E_{6,6} \\\hline
\mathbb{Q}(\cos(\frac{2\pi}{13})-\cos(\frac{3\pi}{13})) & \mathbf{F_{4,4}} \\\hline\hline\hline
\mathbb{Q}_{15}=\mathbb{Q}_{30} & A_{n,14-n},A_{n,29-n},B_{n,16-2n},C_{n,16-n},D_{n,17-2n}, \\
& D_{n,32-2n},E_{6,18},E_{7,12},F_{4,6},G_{2,6} \\\hline
\mathbb{Q}_{16} & A_{n,15-n},B_{3,3},C_{n,9-n},D_{n,18-2n},E_{6,4} \\\hline
\mathbb{Q}_{20} & A_{n,19-n},B_{3,5},B_{4,3},C_{n,11-n},D_{n,22-2n},E_{6,8} \\\hline
\mathbb{Q}_{24} & A_{n,23-n},B_{n,13-2n},C_{n,13-n},D_{n,26-2n},E_{6,12},E_{7,6},G_{2,4} \\\hline\hline\hline
\mathbb{Q}_{11}=\mathbb{Q}_{22} & A_{n,10-n},A_{n,21-n},B_{3,6},B_{4,4},C_{n,12-n},D_{4,5},D_{5,3}, \\
& D_{n,24-2n},E_{6,10},E_{7,4},\mathbf{E_{8,3}},F_{4,2} \\\hline\hline\hline
\mathbb{Q}_{13}=\mathbb{Q}_{26} & A_{n,12-n},A_{n,25-n},B_{n,14-2n},C_{n,14-n}, \\
& D_{n,15-2n},D_{n,28-2n},E_{6,14},E_{7,8} \\\hline
\mathbb{Q}_{21}=\mathbb{Q}_{42} & A_{n,20-n},A_{n,41-n},B_{n,22-2n},C_{n,22-n},D_{n,23-2n},\\
& D_{n,44-2n},E_{6,9},E_{6,30},E_{7,24},E_{8,12},F_{4,12},G_{2,10}  \\\hline
\mathbb{Q}_{28} & A_{n,27-n},B_{n,15-2n},C_{n,15-n},D_{n,30-2n},E_{6,16},E_{7,10},F_{4,5} \\\hline
\mathbb{Q}_{36} & A_{n,35-n},B_{n,19-2n},C_{n,19-n},D_{n,38-2n},E_{6,24},E_{7,18},E_{8,6},F_{4,9},G_{2,8} \\\hline
\mathbb{Q}(\cos(\frac{2\pi}{35})+\cos(\frac{12\pi}{35})) & \mathbf{E_{8,5}}
\\\hline\hline\hline
\text{none} & \text{none} \\\hline\hline\hline
\mathbb{Q}_{17}=\mathbb{Q}_{34} & A_{n,16-n},A_{n,33-n},B_{n,18-2n},C_{n,18-n},D_{n,19-2n},D_{n,36-2n}, \\
& E_{6,5},E_{6,22},E_{7,16},E_{8,4},F_{4,8} \\\hline
\mathbb{Q}_{32} & A_{n,31-n},B_{n,17-2n},C_{n,17-n},D_{n,34-2n},E_{6,20},E_{7,14},F_{4,7} \\\hline
\mathbb{Q}_{40} &  A_{n,39-n},B_{n,21-2n},C_{n,21-n},D_{n,42-2n},E_{6,28},E_{7,22},E_{8,10},F_{4,11}\\\hline
\mathbb{Q}_{48} & A_{n,47-n},B_{n,25-2n},C_{n,25-n},D_{n,50-2n},E_{6,36},E_{7,30},E_{8,18},G_{2,12},F_{4,15} \\\hline
\mathbb{Q}_{60} & A_{n,59-n},B_{n,31-2n},C_{n,31-n},D_{n,62-2n} \\
& E_{6,48},E_{7,42},E_{8,30},F_{4,21},G_{2,16} \\\hline\hline\hline
\mathbb{Q}_{19}=\mathbb{Q}_{38} & A_{n,18-n},A_{n,37-n},B_{n,20-2n},C_{n,20-n},D_{n,21-2n} \\
& D_{n,40-2n},E_{6,7},E_{6,26},E_{7,20},E_{8,8},F_{4,10} \\\hline
\mathbb{Q}_{27}=\mathbb{Q}_{54} & A_{n,26-n},A_{n,53-n},B_{n,28-2n},C_{n,28-n},D_{n,29-2n},D_{n,56-2n} \\
& E_{6,15},E_{6,42},E_{7,9},E_{7,36},E_{8,24},F_{4,18},G_{2,5},G_{2,14} \\\hline
\end{array}
\end{equation*}
    \caption{$\mathcal{C}(\mathfrak{g},k)$ with $[\mathbb{K}_0:\mathbb{Q}]\leq9$ (separation by degree, exceptions in bold)}%
    \label{fig:A}%
\end{figure}

\begin{figure}[H]
\centering
\begin{equation*}
\begin{array}{|c|c|}
\hline \text{Notation} & \text{Meaning} \\\hline\hline
\mathbb{K}_x & \mathbb{Q}(\mathrm{FPdim}(x)), x\in R \\\hline
\mathbb{K}_0 & \mathbb{Q}(\mathrm{FPdim}(R)) \\\hline
\mathbb{K}_1 & \mathbb{Q}(\mathrm{FPdim}(x):x\in R) \\\hline
R_\mathrm{pt} & \text{pointed fusion subring} \\\hline
R_\mathrm{ad} & \text{adjoint fusion subring} \\\hline
U(R) & \text{universal grading group} \\\hline
\mathcal{O}(\mathcal{C}) & \text{iso. classes of simple objects} \\\hline
\end{array}
\,\,\,
\begin{array}{|c|c|}
\hline \text{Notation} & \text{Meaning} \\\hline\hline
\overline{\mathbb{Q}} & \text{algebraic closure of }\mathbb{Q} \\\hline
\mathbb{Q}_n & \mathbb{Q}(\cos(2\pi/n)) \\\hline
\mathbb{L}_X & \mathbb{Q}(s_{Y,X}/s_{\mathbbm{1},X}:Y\in\mathcal{O}(\mathcal{C})) \\\hline
\mathbb{L} & \mathbb{Q}(s_{Y,X}/s_{\mathbbm{1},X}:X,Y\in\mathcal{O}(\mathcal{C})) \\\hline
\xi_n & \exp(2\pi i/n), n\in\mathbb{Z}_{\geq1} \\\hline
h^\vee & \text{dual Coxeter number} \\\hline
\kappa & k+h^\vee, k\in\mathbb{Z}_{\geq1} \\\hline
\end{array}
\end{equation*}
    \caption{Recurring notation ($R$ a fusion ring, $\mathcal{C}$ a fusion/modular tensor category)}%
    \label{fig:X}%
\end{figure}

\end{section}


\begin{section}{Preliminary notions}\label{prelim}

Here we recount the basic notions and fundamental examples of fusion rings and fusion, braided, and modular tensor categories.  We refer the reader to \cite{tcat} for additional categorical details and references, and to the earlier  \cite{fuchs1994fusion,gannon2005modular} for references stressing the combinatorial aspects of the subject.

\begin{subsection}{Fusion rings and categories}
\par A \emph{fusion ring} is a ring $R$ with unit $1\in R$ which is free as a $\mathbb{Z}$-module, and a distinguished basis $1=b_0,\ldots,b_n$ such that
\begin{itemize}
\item[(a)] if $b_ib_j=\sum_{k=0}^{n}c_{ij}^kb_k$ for $0\leq i,j\leq n$, then $c_{ij}^k\in\mathbb{Z}_{\geq0}$, and
\item[(b)] there exists an anti-involution (duality) of $b_0,\ldots,b_n$ given by $b_i\mapsto(b_i)^\ast$ (extended linearly to all of $R$) such that $c_{ij}^0=1$ if $b_i=(b_j)^\ast$ and $c_{ij}^0=0$ otherwise.
\end{itemize}

\begin{example}\label{fusionring}
The ring $\mathbb{Z}G$ for a finite group $G$ with basis $g\in G$ provides an example of a fusion ring with duality $g\mapsto g^{-1}$.  The complex character ring $R(G)$ with a basis of irreducible characters and duality given by complex conjugation gives another example.
\end{example}

\par A \emph{fusion subring} is a pair of fusion rings $S\subset R$  such that the basis of $S$ is contained in the basis of $R$.  Every fusion ring $R$ has at least two distinguished fusion subrings: the \emph{pointed subring} $R_\mathrm{pt}$, whose basis consists of all invertible basis elements of $R$ ($b_j$ with $b_jb_j^\ast=b_0$), and the \emph{adjoint subring} $R_\mathrm{ad}$, whose basis consists of all basis elements of $R$ which appear as summands of $b_jb_j^\ast$ for $0\leq j\leq n$.

A \emph{grading} of a fusion ring $R$ with basis $B:=\{b_j\}_{j\in I}$ for $j$ in a finite index set $I$, is a partition of $B=\cup B_g$ by a finite group $G$ such that if $b_i\in B_g$ and $b_j\in B_h$, then $b_ib_j$ is a $\mathbb{Z}_{\geq0}$-linear combination of basis elements in $B_{gh}$.  Consequently if $b_j\in B_g$, then $b_j^\ast\in B_{g^{-1}}$.  The grading is faithful if $B_g$ is not empty for any $g\in G$.  If a fusion ring $R$ is graded by a group $G$ and $R_g$ is the $\mathbb{Z}$-linear span of the basis elements in $B_g$, one has a decomposition
\begin{equation}
R=\bigoplus_{g\in G}R_g.
\end{equation}
Each fusion ring $R$ has a canonical faithful grading by a universal grading group $U(R)$ whose components are the components of $R$ when decomposed as a based $R_\mathrm{ad}$-bimodule \cite[Section 3.6]{tcat}.  In particular, $R_\mathrm{ad}$ is the trivial component of the universal grading.

\begin{example}
Let $G$ be a finite group with identity element $e\in G$ and recall the fusion rings $\mathbb{Z}G$ and $R(G)$ from Example \ref{fusionring}.  The ring $\mathbb{Z}G$ is \emph{pointed}, i.e. \!$(\mathbb{Z}G)_\mathrm{pt}=\mathbb{Z}G$, while $(\mathbb{Z}G)_\mathrm{ad}$ is trivial.  Hence the universal grading group of $\mathbb{Z}G$ is $G$.  Compare this to $R(G)$.  Let $\chi_0,\ldots,\chi_n$ be the basis for $R(G)$.  Then the basis for $R(G)_\mathrm{pt}$ are those $\chi_i$ with $\chi_i(e)=1$.  One should then verify that the basis of  $R(G)_\mathrm{ad}$ are those $\chi_i$ with $\mathrm{ker}\,\chi_i\supset\mathrm{Z}(G)$ where $\mathrm{Z}(G)$ is the center of $G$, and thus $U(R(G))$ is the group of all irreducible characters of $\mathrm{Z}(G)$. 
\end{example}

Motivating examples of fusion rings arise from the following tractable categories.

\begin{defi}\label{fusioncategory}
A \emph{fusion category} (over $\mathbb{C}$) is a $\mathbb{C}$-linear semisimple rigid tensor category with simple $\otimes$-unit $\mathbbm{1}$, finitely many isomorphism classes of simple objects, and finite-dimensional spaces of morphisms \cite[Definition 4.1.1]{tcat}.
\end{defi}

\par Denote the set of isomorphism classes of simple objects of a fusion category $\mathcal{C}$ by $\mathcal{O}(\mathcal{C})$.  One often refers to the decomposition of tensor products in $\mathcal{C}$ into simple objects, afforded by semisimplicity, as the fusion rules of $\mathcal{C}$.  The rigidity of $\mathcal{C}$ (existence of dualizing objects and maps \cite[Section 2.10]{tcat}) describes an anti-involution of $\mathcal{O}(\mathcal{C})$ denoted by $X\leftrightarrow X^\ast$.

\begin{example}\label{ex:groth}
Let $G$ be a finite group. The category of finite-dimensional $G$-graded $\mathbb{C}$-vector spaces Vec$_G$ is a fusion category.  Elements of $\mathcal{O}(\text{Vec}_G)$ are 1-dimensional $\mathbb{C}$-vector spaces indexed by elements of $G$ and duality $g\mapsto g^{-1}$.  One can also twist this construction by nontrivial associative isomorphisms \cite[Example 2.3.8]{tcat}.  The category Rep$(G)$ of finite-dimensional complex representations of $G$ is also a fusion category with the same duality described in Example \ref{fusionring}.

\par The Grothendieck ring of any fusion category $\mathcal{C}$ is a fusion ring with basis $\mathcal{O}(\mathcal{C})$ and duality given by the rigidity of $\mathcal{C}$ \cite[Proposition 4.5.4]{tcat}.  In this way the fusion rings $\mathbb{Z}G$ and $R(G)$ given in Example \ref{fusionring} arise as the Grothendieck rings of the categories from Example \ref{ex:groth}: Vec$_G$ and Rep$(G)$, respectively.  All the constructions and results related to fusion rings above are then applicable to fusion categories by considering their Grothendieck ring.
\end{example}

\end{subsection}


\begin{subsection}{Modular tensor categories}\label{sec:mod}

A class of fusion categories which are more susceptible to study are the \emph{modular tensor categories}.  Modular tensor categories are fusion categories which are equipped with a spherical structure \cite[Definition 4.7.14]{tcat} and a nondegenerate braiding \cite[Definition 8.1.1]{tcat}.  A spherical structure on a fusion category $\mathcal{C}$ allows one to define traces of endomorphisms (complex scalars), and in particular categorical dimensions $\dim(X)$ \cite[Definition 4.7.11]{tcat} and full twists $\theta_X$ \cite[Definition 8.10.1]{tcat} for each $X\in\mathcal{O}(\mathcal{C})$.

\par The above constructs produce a linear representation of the modular group $SL_2(\mathbb{Z})$ for each modular tensor category, justifying the title.  To describe this representation, we begin by constructing two $|\mathcal{O}(\mathcal{C})|\times|\mathcal{O}(\mathcal{C})|$ matrices: an $\tilde{S}$-matrix, consisting of the traces of the double braidings $\tilde{s}_{X,Y}$ for each $X,Y\in\mathcal{O}(\mathcal{C})$, and a $\tilde{T}$-matrix, a diagonal matrix of the full twists $\theta_X$.  To describe the necessary normalization constants, define
\begin{align}
p:&=\sum_{X\in\mathcal{O}(\mathcal{C})}\theta_X\dim(X)^2\text{ and} \\
\dim(\mathcal{C})&:=\sum_{X\in\mathcal{O}(\mathcal{C})}\dim(X)^2,\label{glob}
\end{align}
and fix any $\gamma\in\mathbb{C}$ such that $\gamma^3=p/\sqrt{\dim(\mathcal{C})}$ where $\sqrt{\dim(\mathcal{C})}$ is the positive root of (totally positive \cite[Note 2.5]{DNO}) $\dim(\mathcal{C})$.  The matrices $S:=(1/\sqrt{\dim(\mathcal{C})})\tilde{S}$ and $T:=\gamma^{-1}\tilde{T}$ then define a linear representation of $SL_2(\mathbb{Z})$ on the generators
\begin{equation}
\left(\begin{array}{cc} 0 & -1\\ 1&0\end{array}\right)\mapsto S\qquad\text{and}\qquad\left(\begin{array}{cc} 1&1 \\ 0&1\end{array}\right)\mapsto T.
\end{equation}
There is a veritable zoo of tools available to study modular tensor categories, but we will mainly use the following Galois action in our arguments.  The $S=(s_{X,Y})_{X,Y\in\mathcal{O}(\mathcal{C})}$ and $T=\mathrm{diagonal}(t_X)_{X\in\mathcal{O}(\mathcal{C})}$ matrices of each modular tensor category $\mathcal{C}$ are defined over a cyclotomic field $\mathbb{Q}(\zeta_N)$ for some $N$th root of unity $\zeta_N$ and come equipped with an action of the Galois group $\mathrm{Gal}(\mathbb{Q}(\zeta_N)/\mathbb{Q})\cong(\mathbb{Z}/N\mathbb{Z})^\times$ \cite{gannoncoste,de1991markov}.  The integer $N\in\mathbb{Z}_{\geq1}$ is the order of $T$ \cite[Proposition 3(a)]{gannonunpub}, and is often referred to as the conductor of the modular tensor category $\mathcal{C}$.  In particular for each $\sigma\in\mathrm{Gal}(\mathbb{Q}(\zeta_N)/\mathbb{Q})$ (integer $\ell$ coprime to $N$) there exists a unique permutation $\hat{\sigma}$ of $\mathcal{O}(\mathcal{C})$ such that for all $X,Y\in\mathcal{O}(\mathcal{C})$,
\begin{equation}
\sigma\left(\dfrac{s_{X,Y}}{s_{\mathbbm{1},Y}}\right)=\dfrac{s_{X,\hat{\sigma}(Y)}}{s_{\mathbbm{1},\hat{\sigma}(Y)}},
\end{equation}
and \cite[Theorem II(iii)]{dong2015congruence}
\begin{equation}
\sigma^2(t_X)=\sigma^2\left(\dfrac{\theta_X}{\gamma}\right)=\dfrac{\theta_{\hat{\sigma(X)}}}{\gamma}=t_{\hat{\sigma}(X)}.
\end{equation}
The invertible simple objects $X\in \mathcal{O}(\mathcal{C}_\mathrm{pt})$ (i.e. \!those objects with $X\otimes X^\ast\cong \mathbbm{1}$) play a fundamental role and are often called \emph{simple currents}. Tensoring with a simple current permutes the simple objects; in particular this makes $\mathcal{O}(\mathcal{C}_\mathrm{pt})$ into a finite abelian group.  Each simple current $X$ defines a grading on $\mathcal{C}$ through $Y\mapsto s_{X,Y}/s_{\mathbbm{1},Y}$.  In this way, the universal grading group of a modular tensor category is naturally isomorphic to the dual of the group $\mathcal{O}(\mathcal{C}_\mathrm{pt})$ \cite[Theorem 6.3]{nilgelaki}.

\begin{example}[Quantum groups at roots of unity]\label{ex:quant1}
Let $\mathfrak{g}$ be a simple complex finite-dimensional Lie algebra.  For each \emph{level} $k\in\mathbb{Z}_{\geq1}$ there exists a modular tensor category $\mathcal{C}(\mathfrak{g},k)$ which can be realized by applying a semisimple quotient construction (see, e.g. \!\cite{postriksemisimple}\cite[Section XI.6]{turaev}) to the representation category of $\mathcal{U}_q(\mathfrak{g})$-modules, where $\mathcal{U}_q(\mathfrak{g})$ is a $q$-deformed enveloping algebra associated to $\mathfrak{g}$ at a root of unity dependent on $k$, or equivalently the category of highest-weight integrable $\hat{\mathfrak{g}}$-modules of level $k$, where $\hat{g}$ is the corresponding affine Lie algebra.  These are also the categories of modules of the rational vertex operator algebras associated to $\hat{\mathfrak{g}}$ and level $k$ \cite{kac2013}.  The categories $\mathcal{C}(\mathfrak{g},k)$ have a rich history originating in the mathematical physics literature \cite{ap} and play a pivotal role in the general study of fusion categories as a large class of examples not arising from the representation theory of finite groups.  We invite the reader to \cite[Section 3.3]{BaKi} or \cite{rowell} for a more substantial survey and history of the subject.

\par Elements of $\mathcal{O}(\mathcal{C}(\mathfrak{g},k))$ are indexed by weights of $\mathfrak{g}$ which lie in the Weyl alcove \cite[Section 4]{Sawin03}, a finite truncation of the classical dominant Weyl chamber.  Although not used in what follows, fusion rules in $\mathcal{C}(\mathfrak{g},k)$ can be computed using the Kac-Walton formula for affine Lie algebras \cite[p. 288]{kac2013}\cite{walton} (\cite[Section 5]{Sawin03} in the language of quantum groups), and duality is contragradience inherited from the classical representation theory \cite[Chapter 6]{hump}.
\end{example}

\end{subsection}

\begin{subsection}{Frobenius-Perron dimension}
\par Oskar Perron \cite{perron} demonstrated in 1907 that square matrices of positive real numbers enjoy the following fundamental property.
\begin{theorem}[Frobenius-Perron]\label{frobper}
If $M$ is a square matrix of positive real numbers, then the spectral radius of $M$ is one of its eigenvalues.
\end{theorem}
\par In particular, such a matrix possesses a positive real eigenvalue.  Later, Georg Frobenius \cite{frobenius} would generalize this result to irreducible square matrices of nonnegative real numbers.  Using the Frobenius-Perron theorem, the concept of dimension can be extended to algebraic structures posessing a well-behaved notion of multiplication such as fusion rings and fusion categories.  We will refer to this real eigenvalue as the \emph{Frobenius-Perron eigenvalue} of $M$.

\begin{defi}\label{fpdim}
Let $R$ be a fusion ring with basis $b_0,\ldots,b_n$.  The \emph{Frobenius-Perron dimension} $d_i$ (or FPdim$(b_i))$ is the Frobenius-Perron eigenvalue of the fusion matrix $(c_{ij}^k)_{0\leq j,k\leq n}$ for each $0\leq i\leq n$.  We extend $\mathrm{FPdim}$ linearly to all of $R$.
\end{defi}

Proposition 3.3.6 of \cite{tcat} states $\fpdim$ is a ring homomorphism FPdim$:R\to\mathbb{C}$.  Frobenius-Perron dimensions of objects of a fusion category $\mathcal{C}$ are defined analogously by considering the Grothendieck ring of $\mathcal{C}$ as a fusion ring.  One should also note that the $\fpdim$ of any object in $\mathcal{C}$ is an algebraic integer \cite[Proposition 3.3.4(1)]{tcat}.  Lastly define the \emph{Frobenius-Perron dimension} of a fusion ring $R$ (similarly, a fusion category $\mathcal{C}$) as $\fpdim(R):=\sum_{j=0}^nd_j^2$, analagous to the \emph{global dimension} $\dim(\mathcal{C})$ defined in Equation (\ref{glob}).

\begin{example}[Quantum groups at roots of unity continued]\label{ex:quant2}
Let $\mathfrak{h}\subset\mathfrak{g}$ be a Cartan subalgebra and $\langle.\,,.\rangle$ be the invariant form on $\mathfrak{h}^\ast$ normalized so $\langle\alpha,\alpha\rangle=2$ for short roots $\alpha\in\Delta$, using the notation of \cite{hump}.  We will use $\rho$ to represent half the sum of the positive roots $\alpha\in\Delta^+$, $h^\vee$ (Figure \ref{tab:kappa}) to be the dual Coxeter number of $\mathfrak{g}$, and $m$ to be the least positive integer such that $m\rho$ is a sum of coroots $\alpha^\vee\in\Delta^\vee$.  The \emph{quantum Weyl dimension formula} \cite[Equation 3.3.5]{BaKi} states that for all $\lambda\in\mathcal{O}(\mathcal{C}(\mathfrak{g},k))$,
\begin{equation}\label{weyl}
\dim(\lambda)=\prod_{\alpha\in\Delta^+}\dfrac{[\langle\alpha,\lambda+\rho\rangle]_q}{[\alpha,\rho]_q}
\end{equation}
where $[n]_q=(q^n-q^{-n})/(q-q^{-1})$ with $q=\exp(\pi i/(m(k+h^\vee)))$.
\begin{figure}[H]
\centering
\[\begin{array}{|c|c|c|c|c|c|c|c|c|c|}
\hline\text{Type} & A_n & B_n & C_n & D_n & E_6 & E_7 & E_8 & F_4 & G_2 \\\hline
h^\vee& n+1 & 2n-1 & n+1 & 2n-2 & 12 & 18 & 30 & 9 & 4 \\\hline
\end{array}\]
    \caption{Dual Coxeter numbers $h^\vee$}%
    \label{tab:kappa}%
\end{figure}
Although we have described the quantum Weyl dimension formula using categorical dimensions of objects in $\mathcal{C}(\mathfrak{g},k)$, we can equivalently study Frobenius-Perron dimension with this formula because $\mathcal{C}(\mathfrak{g},k)$ are pseudounitary.
\begin{defi}
A fusion category is \emph{pseudounitary} if $\dim(\mathcal{C})=\fpdim(\mathcal{C})$.
\end{defi}
Each $\mathcal{C}(\mathfrak{g},k)$ is pseudounitary (moreover unitary) and so the spherical structure yielding the positive dimensions above is the canonical spherical structure \cite[Proposition 8.23]{DNO} for which $\dim(X)=\fpdim(X)$ for all $X\in\mathcal{O}(\mathcal{C}(\mathfrak{g},k))$.  To choose roots of unity $q$ different from those specified above is to choose an alternate spherical structure which may produce negative categorical dimensions, though the Frobenius-Perron dimensions are unaffected by these alternations.
\end{example}

\end{subsection}

\end{section}


\begin{section}{The main lemma and a faithful grading}\label{sec:main}

\begin{subsection}{The main lemma}

\begin{lem}\label{theorem}
Let $R$ be a fusion ring.  If $x_1,x_2\in R$ are $\mathbb{Z}_{\geq0}$-linear combinations of basis elements of $R$, then $\mathrm{FPdim}(x_1),\mathrm{FPdim}(x_2)\in\mathbb{Q}(\mathrm{FPdim}(x_1+x_2))$.
\end{lem}

\begin{proof}
Denote the basis of the fusion ring $R$ by $b_0,\ldots,b_n$.  Assume $x_i\in R$ for $i=1,2$ with $x_i=\sum_{j=0}^na_{i,j}b_j$ and $a_{i,j}\in\mathbb{Z}_{\geq0}$ for all $0\leq j\leq n$. Set $c_j:=a_{1,j}+a_{2,j}$ so that for all $\sigma\in\text{Gal}(\overline{\mathbb{Q}}/\mathbb{K}_{x_1+x_2})$, $\mathrm{FPdim}(x_1+x_2)=|\sigma(\mathrm{FPdim}(x_1+x_2))|$ and by the triangle inequality,
\begin{equation}
\mathrm{FPdim}(x_1+x_2)=\left|\sum_{j=0}^nc_j\sigma(d_j)\right|\leq\sum_{j=0}^nc_j|\sigma(d_j)|\leq\sum_{j=0}^nc_jd_j.\label{galois}
\end{equation}
Hence $\sum_{j=0}^nc_j(d_j-|\sigma(d_j)|)=0$.  But $c_j(d_j-|\sigma(d_j)|)\geq0$ for all $0\leq j\leq n$ as $d_j$ is maximal amongst its Galois conjugates in modulus by definition.  Therefore we have $c_j(d_j-|\sigma(d_j)|)=0$ and if $c_j\neq0$, then $|\sigma(d_j)|=d_j$.  The consequence of the triangle inequality being an equality is that $c_j\sigma(d_j)$ are all a positive real scalar multiple of one another (as they sum to a positive real number).  Hence when $c_j\neq0$, $\sigma(d_j)=d_j$.  One has $c_j\neq0$ if and only if $b_j$ appears in the decomposition of $x_1$ or $x_2$.  Thus $\fpdim(x_1)$ and $\fpdim(x_2)$ are fixed by all $\sigma\in\text{Gal}(\overline{\mathbb{Q}}/\mathbb{K}_{x_1+x_2})$ and therefore lie in $L$ as $\overline{\mathbb{Q}}/\mathbb{K}_{x_1+x_2}$ is a Galois extension.
\end{proof}

\begin{note}
We note that the integrality in Lemma \ref{theorem} is not crucial, but the positivity is.  One could safely replace $\mathbb{Z}$ with $\mathbb{Q}$, though we currently see no compelling reason to do so.
\end{note}

\begin{theo}\label{zerocor}
Let $R$ be a fusion ring.  If $S\subset R$ is a fusion subring, then $\mathrm{FPdim}(S)\in\mathbb{K}_0$.
\end{theo}

\begin{proof}
Let $b_0,\ldots,b_n$ be the basis of $R$.  Our result  then follows from Lemma \ref{theorem} with $x_1=\sum_{b_j\in S}b_jb_j^\ast$ and $x_2=\sum_{b_j\not\in S}b_jb_j^\ast$.
\end{proof}

\begin{theo}{\textnormal{(c.f. \!\cite[Theorem 7.3]{MR3108081})}}\label{cor:second}
Let $R$ be a fusion ring with basis $b_0,\ldots b_n$.  For all $0\leq j\leq n$, $d_j^2\in\mathbb{K}_0$.
\end{theo}

\begin{proof}
Use Lemma \ref{theorem} with $x_1=\sum_{k\neq j}b_kb_k^\ast$ and $x_2=b_jb_j^\ast$.
\end{proof}

Recall that the Frobenius-Perron dimension of a fusion ring is \emph{totally} real, which is to say if $R$ is a fusion ring, for all $\sigma\in\mathrm{Gal}(\overline{\mathbb{Q}}/\mathbb{Q})$, $\sigma(\mathrm{FPdim}(R))\in\mathbb{R}$ (see the comments following \cite[Definition 3.3.12]{tcat}).  The same is true for $\dim(\mathcal{C})$ for any fusion category $\mathcal{C}$, and more generally the \emph{squared norm} $|X|^2$ \cite[Section 2.1]{ENO} for any simple $X\in\mathcal{O}(\mathcal{C})$.  When $\mathcal{C}$ is equipped with a spherical structure, $\dim(X)^2=|X|^2$ \cite[Corollary 2.10]{ENO}, hence $\dim(X)$ is totally real for any $X\in\mathcal{O}(\mathcal{C})$ as well.  Without sphericality, the categorical dimensions of simple objects may lie outside of $\mathbb{R}$ \cite[Remark 2.11]{ENO}.

\begin{theo}\label{cor:new}
If $R$ is a fusion ring, then $\mathrm{FPdim}(x)$ is totally real for all $x\in R$.
\end{theo}

\begin{proof}
It suffices to prove this for the $\mathrm{FPdim}$ of basis elements $b_0,\ldots,b_n$.  If $d_j\not\in\mathbb{K}_0$, Proposition \ref{cor:second} states $\mathbb{K}_0(d_j)/\mathbb{K}_0$ is a real quadratic extension of a real field, and thus $d_j$ is totally real.
\end{proof}

\begin{theo}\label{cor:first}
If $R$ is a fusion ring and $\mathbb{K}$ an algebraic number field, then
\begin{equation}
R_\mathbb{K}:=\{x\in R:\mathrm{FPdim}(x)\in \mathbb{K}\}
\end{equation}
is a fusion subring.
\end{theo}

\begin{proof}
The set $R_\mathbb{K}$ is closed under addition and multiplication as $\fpdim$ is a ring homomorphism, closed under duality as $\fpdim(x)=\fpdim(x^\ast)$ \cite[Proposition 3.3.9]{tcat} and Lemma \ref{theorem} implies $R_\mathbb{K}$ is closed under positive additive decomposition. Moreover the basis elements of $R$ with $\fpdim$ in $\mathbb{K}$ span $R_\mathbb{K}$. 
\end{proof}

We say that a fusion ring $R$ is \emph{multiplicatively generated} by $x\in R$ if each basis element of $R$ is a summand of $x^n$ for some $n\in\mathbb{Z}_{\geq1}$.

\begin{theo}
Let $x\in R$ be an element of a fusion ring.  If $R$ is multiplicatively generated by $x$, then $\mathrm{FPdim}(y)\in\mathbb{K}_x$ for all $y\in R$.
\end{theo}

\begin{proof}
The result follows from Lemma \ref{theorem} as every $y\in R$ is a summand of $x^n$ for some $n\in\mathbb{Z}_{\geq1}$.
\end{proof}

\end{subsection}


\begin{subsection}{A dimensional grading}\label{grading}

The fields generated by Frobenius-Perron dimensions of elements in a fusion ring $R$ are intimately tied to gradings of $R$.  In \cite[Theorem 3.10]{nilgelaki}, it was emphasized that every fusion ring with integer $\mathrm{FPdim}$ (\emph{weakly integral}) has a faithful dimensional grading by an elementary abelian 2-group.   This is true of all fusion rings.

\begin{theo}{\textnormal{(c.f. \!\cite[Theorem 7.5]{MR3108081})}}\label{intrograd}
If $R$ is a fusion ring, $R$ is faithfully graded by $\mathrm{Gal}(\mathbb{K}_1/\mathbb{K}_0)$.
\end{theo}

\begin{proof}
Let $B:=\{b_0,\ldots,b_n\}$ be the basis of $R$.  Define a relation on $B$ by $b_i\sim b_j$ if and only if $\mathbb{K}_0(d_i)=\mathbb{K}_0(d_j)$.  This is clearly an equivalence relation and so the set of equivalence classes $\{B_g\}_{g\in G}$ partitions $B$ for some index set $G$.  Each equivalence class is nonempty, implying the resulting grading will be faithful.  For the purposes of notation we index the equivalence class of $b_0$ by $B_0$, and set $\mathbb{K}_i:=\mathbb{K}_0(d_i)$.  Assume $b_i\in B_r$ and $b_j\in B_s$.  If $d_id_j\in\mathbb{K}_0$ then $b_ib_j$ is a $\mathbb{Z}_{\geq0}$-linear combination of elements of $B_0$ by Lemma \ref{theorem}.  Otherwise $[\mathbb{K}_0(d_id_j):\mathbb{K}_0]=2$ as $(d_id_j)^2\in\mathbb{K}_0$ by Proposition \ref{cor:second}.  Hence all $b_k$ with $c_{ij}^k\neq0$ have $d_k\in\mathbb{K}_0$ or $d_k\in\mathbb{K}_0(d_id_j)\setminus\mathbb{K}_0$ by Lemma \ref{theorem}.  We aim to prove the former cannot happen.  To this end, there exists $\sigma\in\mathrm{Gal}(\mathbb{K}_0(d_id_j)/\mathbb{K}_0)$ with $\sigma(\alpha)=-\alpha$ for all $\alpha\in\mathbb{K}_0(d_id_j)\setminus\mathbb{K}_0$ because $\mathbb{K}_0(d_id_j)/\mathbb{K}_0$ is a quadratic Galois extension.  Hence
\begin{equation}
0=d_id_j+\sigma(d_id_j)=2\sum_{d_k\in\mathbb{K}_0}c_{ij}^k d_k.
\end{equation}
As $d_k>0$ for $b_k\in B$, we must have $c_{ij}^k=0$ whenever $d_k\in\mathbb{K}_0$.  There exists at least one $b_k$ with $c_{ij}^k\neq0$ hence $\mathbb{K}_0(d_id_j)=\mathbb{K}_t$ for some $t\in G$, and we have shown $b_ib_j\in B_rB_s$ decomposes into a $\mathbb{Z}_{\geq0}$-linear combination of elements of $B_t$.  Moreover $R$ is faithfully graded by the set $G$ and the above construction defines an associative binary operation on $G$ (i.e. \!$rs:=t$).  Note that for all $b_i\in B$, $b_0b_i=b_i$, and $b_i^2$ is a $\mathbb{Z}$-linear combination of elements of $B_0$ by Lemma \ref{theorem}, so $G$ is a finite group in which every element is order at most 2.  Therefore $G$ is an elementary abelian 2-group.  As $\mathbb{K}_1/\mathbb{K}_0$ is a Galois extension and the group $G$ is describing the set of intermediate subfields of $\mathbb{K}_1/\mathbb{K}_0$, it is clear that $G\cong\mathrm{Gal}(\mathbb{K}_1/\mathbb{K}_0)$.
\end{proof}

If $R$ is a fusion ring, there exists a surjective group homomorphism \cite[Corollary 3.7]{nilgelaki}
\begin{equation}\label{pi}
\pi_{FP}:U(R)\to\mathrm{Gal}(\mathbb{K}_1/\mathbb{K}_0).
\end{equation}
We refer to this faithful grading as the \emph{dimensional grading} for brevity.

\begin{theo}\label{cyclic}
If $R$ is a fusion ring and $U(R)$ is cyclic, then $[\mathbb{K}_1:\mathbb{K}_0]=1,2$.  Moreover if $|U(R)|$ is odd, $\mathbb{K}_0=\mathbb{K}_1$.
\end{theo}

\begin{proof}
In the case $U(R)$ is cyclic, the image of $\pi_{FP}$ is a subgroup of $\mathbb{Z}/2\mathbb{Z}$ as any quotient of a cyclic group is cyclic.  In the case $|U(R)|$ is odd, the image of $\pi_{FP}$ is trivial.
\end{proof}

\begin{note}\label{note1}
Assume the image of $\pi_{FP}$ is $\mathbb{Z}/2\mathbb{Z}$.  If $b_j\in R_g$ for some $g\in U(R)$, then $\mathbb{K}_j=\mathbb{K}_1$ if and only if $g$ has even order.  This will be applied meaningfully in Section \ref{sec:quantum}.
\end{note}

\begin{example}\label{exe72}
We will illustrate that the dimensional grading $\pi_{FP}$ does not capture the entire complexity of the fields of Frobenius-Perron dimensions of objects in a fusion ring $R$, even when the hypotheses of Proposition \ref{cyclic} are true.  For example let $S$ be the rank 3 fusion ring $\{1_S,X_S,Y_S\}$ with $X_S^2=1_S+Y_S$, $Y_S^2=1$, and $X_SY_S=Y_SX_S=X_S$, and $T$ be the rank 2 fusion ring $\{1_T,X_T\}$ with $X_T^2=1+X_T$.  The product $R:=S\times T$ is a rank 6 fusion ring with $\pi_R:U(R)\to\mathbb{Z}/2\mathbb{Z}$ an isomorphism.  This grading cannot detect the field $\mathbb{K}_{X_S1_T}=\mathbb{Q}(\sqrt{2})$ because it is a degree 2 extension of a proper subfield of $\mathbb{K}_0$, not of $\mathbb{K}_0$ itself.  It would be of great interest to find an example of this phenomenon which is not a product (see Section \ref{discussion}).
\begin{figure}[H]
\centering
\begin{tikzpicture}[scale=1.2]
\node (a) at (0,0) {$\mathbb{K}_{1_S1_T}=\mathbb{K}_{Y_S1_T}=\mathbb{Q}$};
\node (c) at (5,0) {$\mathbb{K}_{X_S1_T}=\mathbb{Q}(\sqrt{2})$};
\node (d) at (10,0) {$\mathbb{K}_{X_SX_T}=\mathbb{Q}(\sqrt{2},\sqrt{5})$};
\node (b) at (5,-1) {$\mathbb{K}_0=\mathbb{K}_{1_SX_T}=\mathbb{K}_{Y_SX_T}=\mathbb{Q}(\sqrt{5})$};
\draw[->] (a) -- (c);
\draw[->] (c) -- (d);
\draw[->] (a) -- (b);
\draw[->] (b) -- (d);
\end{tikzpicture}
    \caption{$\fpdim$ fields generated by $x\in R$ ($\to$ is inclusion of subfields)}%
    \label{fig:appa}%
\end{figure}

\end{example}

\end{subsection}

\end{section}


\begin{section}{A multiplicative central charge bound}\label{seccharge}

If $\mathcal{C}$ is a modular tensor category, recall the Galois action on the modular data of $\mathcal{C}$ and the definition of $\gamma\in\mathbb{C}$ from Section \ref{sec:mod} such that $\xi(\mathcal{C}):=\gamma^3=p/\sqrt{\dim(\mathcal{C})}$ which is known as the \emph{multiplicative central charge} of $\mathcal{C}$.  The multiplicative central charge $\xi$ is related to the classical central charge $c$ via $\xi=\exp(2\pi i c/8)$.  For example, the multiplicative central charge of $\mathcal{C}(\mathfrak{g},k)$ \cite{kac2013} is given by 
\begin{equation}\label{quantcharge}
\xi(\mathcal{C}(\mathfrak{g},k))=\exp\left(\dfrac{2\pi i}{8}\cdot\dfrac{k\dim\mathfrak{g}}{k+h^\vee}\right).
\end{equation}

Proposition 3(b) of \cite{gannonunpub} and Theorem 6.10 of \cite{dong2015congruence} tell us that if $\mathcal{C}$ is pseudounitary and $\mathbb{K}_1=\mathbb{Q}$, then $\xi(\mathcal{C})^8=1$, or equivalently $c\in\mathbb{Z}$.  Lemma 1.1 permits us to generalize the proof to arbitrary dimension fields.

\par Let $\mathcal{C}$ be a modular tensor category.  For each $X\in\mathcal{O}(\mathcal{C})$, we define
\begin{equation}\label{eqx}
\mathbb{L}_X:=\mathbb{Q}\left(\dfrac{s_{Y,X}}{s_{\mathbbm{1},X}}:Y\in\mathcal{O}(\mathcal{C})\right).
\end{equation}
In particular, $\mathbb{L}_\mathbbm{1}=\mathbb{Q}(\dim(X):X\in\mathcal{O}(\mathcal{C}))$ and $\mathbb{L}_\mathbbm{1}=\mathbb{K}_1$ when $\mathcal{C}$ is pseudounitary.  When $\mathcal{C}$ is not pseudounitary, we still have $\mathbb{K}_1=\mathbb{L}_X$ for some $X\in\mathcal{O}(\mathcal{C})$.  For all $n\in\mathbb{Z}_{\geq1}$ define $f(n)$ to be the largest $m\in\mathbb{Z}_{\geq2}$ such that the exponent of $(\mathbb{Z}/m\mathbb{Z})^\times$ is equal to $n$.

\begin{theo}\label{th:charge}
Assume $\mathcal{C}$ is a modular tensor category, $X\in\mathcal{O}(\mathcal{C})$ and $N\in\mathbb{Z}_{\geq1}$ is the exponent of $\mathrm{Gal}(\mathbb{L}_X/\mathbb{Q})$.  Then $t_X$ is a root of unity of order dividing $f(2N)$.
\end{theo}

\begin{proof}
Let $\sigma\in\mathrm{Gal}(\overline{\mathbb{Q}}/\mathbb{L}_X)$.  Then $s_{Y,X}/s_{\mathbbm{1},X}=s_{Y,\hat{\sigma}(X)}/s_{\mathbbm{1},\hat{\sigma}(X)}$.  By unitarity of $S$,
\begin{equation}
\dfrac{1}{s_{\mathbbm{1},X}}=\sum_{Y\in\mathcal{O}(\mathcal{C})}\overline{s_{X,Y}}\dfrac{s_{Y,X}}{s_{\mathbbm{1},X}}=\sum_{Y\in\mathcal{O}(\mathcal{C})}\overline{s_{X,Y}}\dfrac{s_{Y,\hat{\sigma}(X)}}{s_{\mathbbm{1},\hat{\sigma}(X)}}=\dfrac{\delta_{X,\hat{\sigma}(X)}}{s_{\mathbbm{1},\hat{\sigma}(X)}}.
\end{equation}
Hence $X=\hat{\sigma}(X)$ for all such $\sigma$.  Moreover $\sigma^2(t_X)=t_{\hat{\sigma}(X)}=t_X$ as well.  If $\sigma\in\mathrm{Gal}(\overline{\mathbb{Q}}/\mathbb{Q})$, $\sigma^N$ fixes $\mathbb{L}_X$, hence $\sigma^{2N}(t_X)=t_X$ and so the exponent of $\mathrm{Gal}(\mathbb{Q}(t_X)/\mathbb{Q})$ must divide $2N$.  Then $t_X$ is a root of unity of order which divides $f(2N)$ by definition of $f(n)$.
\end{proof}

\begin{cor}\label{bettercor}
Let $\mathcal{C}$ be a modular tensor category and let $N\in\mathbb{Z}_{\geq1}$ be the exponent of $\mathrm{Gal}(\mathbb{L}_\mathbbm{1}/\mathbb{Q})$.  Then $\xi(\mathcal{C})$ is a root of unity of order dividing $f(2N)/3$.
\end{cor}

\begin{example}
We will compute $f(2N)$ for the first couple $N\in\mathbb{Z}_{\geq1}$ to illustrate the general method (Lemma \ref{comp}).  If $\lambda(m)$ is the exponent of $(\mathbb{Z}/m\mathbb{Z})^\times$ (sometimes referred to the Carmichael lambda function), and $m=\prod_jp_j^{a_j}$ is the prime decomposition of $m$,
\begin{equation}
\lambda(m)=\mathrm{lcm}\{\lambda(p_j^{a_j}):a_j>0\},
\end{equation}
with
\begin{equation}
\lambda(p_j^{a_j})=\left\{\begin{array}{ccc}p_j^{a_j-1}(p_j-1) & : & p_j^{a_j}=2,4\text{ or is odd} \\ p_j^{a_j-2} & : & \text{else}\end{array}\right..
\end{equation}
Thus if $\lambda(m)=2N$, then $\lambda(p_j^{a_j})$ divides $2N$ for all primes $p_j$ dividing $m$.

\par For example if $\lambda(m)=2$, then $\lambda(p_j^{a_j})=1,2$ for each prime power divisor of $m$.  Thus $m$ is not divisible by an odd prime $p>3$ and thus only by $3^1$.  Lastly, $m$ could only be divisible by $2^a$ for $a\leq3$.  Moreover $f(2)=2^3\cdot3^1=24$.  Or if $\lambda(m)=4$, then $\lambda(p_j^{a_j})=1,2,4$ for each prime power divisor of $m$.  Thus $m$ is not divisible by an odd prime $p>5$ and only by $3^1$ and $5^1$.  Lastly, $m$ could only be divisible by $2^a$ for $a\leq4$.  Moreover $f(4)=2^4\cdot3^1\cdot5^1=240$.  The first 10 values of $f(2N)/3$ are given below in Figure \ref{fig:K}.
\end{example}

\begin{lem}\label{comp}
If $N=2^{a_0}\prod_{j\neq0}p_j^{a_j}\in\mathbb{Z}_{\geq1}$, then $f(2N)=2^{a_0+3}\prod_{j\neq0}q_j$ with
\begin{equation}
q_j=\left\{\begin{array}{lcr}1 & : & (p_j-1)\nmid 2N \\ p_j & : &(p_j-1)\mid 2N\text{ and }p_j\nmid 2N \\   p_j^{a_j+1} & : &(p_j-1)\mid 2N\text{ and }p_j\mid 2N \end{array}\right..
\end{equation}
\end{lem}

\begin{figure}[H]
\centering
\begin{equation*}
\begin{array}{|c|c|c|c|c|c|c|c|c|c|c|}
\hline N & 1 & 2 & 3 & 4 & 5 & 6 & 7 & 8 & 9 & 10 \\\hline
f(2N)/3 & 8 & 80 & 56 &  160 &  88 &  7280 &  8 &  5440 &  1064 &   880 \\\hline
\end{array}
\end{equation*}
    \caption{Order bound for $\xi(\mathcal{C})$ based on exponent $N$ of $\mathrm{Gal}(\mathbb{L}_{\mathbbm{1}}/\mathbb{Q})$}%
    \label{fig:K}%
\end{figure}

\begin{note}\label{chargenote}
When $\mathcal{C}$ is pseudounitary, $\mathbb{L}_\mathbbm{1}=\mathbb{K}_1$ hence the exponent of $\mathrm{Gal}(\mathbb{L}_\mathbbm{1}/\mathbb{Q})$ is equal to, or twice the exponent of $\mathrm{Gal}(\mathbb{K}_0/\mathbb{Q})$ as $\mathrm{Gal}(\mathbb{K}_1/\mathbb{K}_0)$ is an elementary abelian 2-group (refer to the proof of Proposition \ref{intrograd}).  Denote the exponent of $\mathrm{Gal}(\mathbb{K}_0/\mathbb{Q})$ as $M\in\mathbb{Z}_{\geq1}$ for brevity.  Therefore, when $\mathcal{C}$ is pseudounitary, Corollary \ref{bettercor} implies $\xi(\mathcal{C})$ is a root of unity of order dividing $f(2M)/3$ when $\mathrm{Gal}(\mathbb{K}_1/\mathbb{K}_0)$ is trivial.  Otherwise $\xi(\mathcal{C})$ is a root of unity of order dividing $2f(2M)/3$. 
\end{note}

\begin{theo}\label{lemprime}
Let $\mathcal{C}$ be a modular tensor category such that $\mathbb{L}_\mathbbm{1}=\mathbb{K}_1$ (e.g. \!$\mathcal{C}$ pseudounitary) and $N\in\mathbb{Z}_{\geq1}$ the exponent of $\mathrm{Gal}(\mathbb{L}_\mathbbm{1}/\mathbb{Q})$.  If $N=p^b$ for some prime $p$ and $b\in\mathbb{Z}_{\geq1}$, and $2p^a+1$ is not prime for any $1\leq a\leq b$, then $\xi(\mathcal{C})^{16}=1$.  Furthermore, if $\mathbb{L}_\mathbbm{1}=\mathbb{K}_1=\mathbb{K}_0$, then $\xi(\mathcal{C})^8=1$.
\end{theo}

\begin{proof}
If $\lambda(m)=2N$ and $m=\prod_jp_j^{a_j}$ as above, then $\lambda(p_j^{a_j})=1,2,p^a,2p^a$ for each prime power divisor of $m$ where $a\in\mathbb{Z}_{\geq1}$ and $1\leq a\leq b$.  Note that $3$ is the only odd prime divisor of $m$ as $p_j-1\neq p^a,2p^a$ for any primes $p_j>3$ by assumption, and $2^3$ is the largest power of $2$ dividing $m$.  Hence $f(2N)=24$ and Corollary \ref{bettercor} implies our result with the comments in Note \ref{chargenote}.
\end{proof}

\begin{note}
Any prime $p\equiv1\pmod{3}$ has $2p^a+1\equiv 0\pmod{3}$ for all $a\in\mathbb{Z}_{\geq1}$, so there are infinitely-many prime powers $N=p^b$ satisfying the hypotheses of Proposition \ref{lemprime}.
\end{note}

\par Proposition \ref{lemprime} in the case $\mathbb{L}_1=\mathbb{K}_1=\mathbb{K}_0$ vastly generalizes the fact that the multiplicative central charge of an integral modular tensor category is an eighth root of unity \cite[Proposition 6.7(ii)]{dong2015congruence} (see also \cite{sommerhauser2013central}).  For readers more familar with the classical (additive) notion of central charge, Proposition \ref{lemprime} implies that the central charge of any modular tensor category of the above type is an integer.

\begin{example}
The bound given by Corollary \ref{bettercor} is tight for $N=1,2,3$, as
\begin{align}
\xi(\mathcal{C}(A_1,2))&=\exp(2\pi i\cdot3/16), \\
\xi(\mathcal{C}(A_1,6)\boxtimes\mathcal{C}(A_2,7))&=\exp(2\pi i\cdot157/160),\text{ and} \\
\xi(\mathcal{C}(A_1,2)\boxtimes\mathcal{C}(A_1,5))&=\exp(2\pi i\cdot51/112).
\end{align}
We do not claim this bound is tight for arbitrary exponents $N\in\mathbb{Z}_{\geq1}$ as it is not clear this would be possible with currently known examples (see Section \ref{discussion}).
\end{example}

\end{section}

\begin{section}{Witt subgroups via dimension fields}

\par The study of algebras in fusion categories is an active area of research due, in part, to its connection to module categories over fusion categories \cite{Ostrik2003}.  Of particular interest in braided fusion categories are \emph{connected \'etale algebras} \cite[Definition 3.1]{DMNO} (sometimes referred to as quantum subgroups), or algebras $A$ in braided fusion categories $\mathcal{C}$ such that $\mathcal{C}_A$, the category of $A$-module objects in $\mathcal{C}$, is a fusion category in its own right.  The $A$-modules whose $A$-bimodule action is indifferent to the braiding of $\mathcal{C}$ form a braided fusion subcategory of $\mathcal{C}_A$ which is nondegenerately braided if and only if $\mathcal{C}$ is (see, for instance, \cite[Section 3]{DMNO}\cite{KiO}).  We denote this fusion subcategory as $\mathcal{C}_A^0\subset\mathcal{C}_A$ and refer to such $A$-modules as local (or dyslectic).  In the language of vertex operator algebras, if $\mathcal{C}$ is the modular tensor category of $\mathcal{V}$-modules, then $\mathcal{C}_A^0$ is the category of $\mathcal{W}$-modules where $\mathcal{W}=A$ is an extension of $\mathcal{V}$.

\par The following definition organizes nondegenerately braided fusion categories by collecting them into equivalence classes of categories which are related by the above construction.

\begin{defi}
Let $\mathcal{C}$ and $\mathcal{D}$ be nondegenerately braided fusion categories.  We say $\mathcal{C}$ and $\mathcal{D}$ are \emph{Witt equivalent} if there exist connected \'etale algebras $A\in\mathcal{C}$ and $B\in\mathcal{D}$ such that $\mathcal{C}_A^0\simeq\mathcal{D}_B^0$ is a braided equivalence.
\end{defi}

The name Witt equivalence is justified as this is an equivalence relation on nondegenerately braided fusion categories, and the equivalence classes of pointed modular tensor categories correspond to the classical Witt group of finite abelian groups with nondegenerate quadratic forms.  We denote Witt equivalence classes by $[\mathcal{C}]$ which form an abelian group $\mathcal{W}$ under Deligne tensor product, i.e. \!$[\mathcal{C}][\mathcal{D}]:=[\mathcal{C}\boxtimes\mathcal{D}]$.  Relations in $\mathcal{W}$ among the classes $[\mathcal{C}(\mathfrak{g},k)]$ in the case $\mathrm{rank}(\mathfrak{g})\leq2$ have been completely classified \cite{DNO,schopieray2017,schopieraysimple}.

\begin{lem}\label{lem:witt}
If $A$ is a connected \'etale algebra in a braided fusion category $\mathcal{C}$, then $\mathrm{FPdim}(\mathcal{C}_A)$ and $\mathrm{FPdim}(\mathcal{C}^0_A)$ lie in $\mathbb{K}_0$.
\end{lem}

\begin{proof}
The fusion category of $A$-bimodules $_A\mathcal{C}_A$ is an indecomposable $\mathcal{C}$-module category with $\mathrm{FPdim}(\mathcal{C})=\mathrm{FPdim}(_A\mathcal{C}_A)$ \cite[Corollary 8.14]{DNO}.  And if one observes $\mathcal{C}^0_A\subset\mathcal{C}_A$ are identified with fusion subcategories of $_A\mathcal{C}_A$, the result follows from Proposition \ref{zerocor}.
\end{proof}

\begin{lem}\label{lem:witt2}
If $A$ is a connected \'etale algebra in a nondegenerately braided fusion category $\mathcal{C}$, then $\mathrm{FPdim}(A)\in\mathbb{K}_0$.
\end{lem}

\begin{proof}
The statement follows from Lemma \ref{lem:witt} because $\mathrm{FPdim}(A)=\mathrm{FPdim}(\mathcal{C}_A)/\mathrm{FPdim}(\mathcal{C}^0_A)$ provided that $\mathcal{C}$ is nondegenerate \cite[Corollary 3.32]{DMNO}.
\end{proof}

\begin{note}
Moreso, $\mathrm{FPdim}(A)$ is a totally positive algebraic $d$-number \cite{codegrees}.
\end{note}

\begin{example}\label{tannak}
The most elementary examples of connected \'etale algebras come from \emph{Tannakian categories} i.e. \!braided fusion categories braided equivalent to $\mathrm{Rep}(G)$ for a finite group $G$.  The regular algebra of $\mathrm{Rep}(G)$ has a canonical structure of a connected \'etale algebra \cite[Example 2.8]{DMNO}.
\end{example}

\par There is a well-defined map $\Phi:\mathcal{W}\to\{\text{algebraic number fields}\}$, where
\begin{equation}
[\mathcal{C}]\mapsto\bigcap_{[\mathcal{D}]=[\mathcal{C}]}\mathbb{Q}(\mathrm{FPdim}(\mathcal{D})).
\end{equation}

\par We say $\mathcal{C}$ is \emph{completely anisotropic} if $\mathcal{C}$ contains no nontrivial connected \'etale algebras.  For example, for a given $\mathfrak{g}$, $\mathcal{C}(\mathfrak{g},k)_A^0$ is completely anisotropic for all but finitely many levels $k$ where $A$ is the regular algebra of a maximal Tannakian fusion subcategory of $\mathcal{C}(\mathfrak{g},k)$.

Each Witt equivalence class contains a unique completely anisotropic representative up to braided equivalence \cite[Theorem 5.13]{DMNO}, so one may use the following to compute $\Phi([\mathcal{C}])$ in examples, though the definition of $\Phi$ does not rely on it.

\begin{lem}\label{int}
If $\mathcal{C}$ is a completely anisotropic nondegenerately braided fusion category, then
\begin{equation}\label{eqwitt}
\Phi([\mathcal{C}])=\mathbb{Q}(\mathrm{FPdim}(\mathcal{C})).
\end{equation}
\end{lem}

\begin{proof}
It is clear the left-hand side of (\ref{eqwitt}) is contained in the right.  If $[\mathcal{D}]=[\mathcal{C}]$ and $\mathcal{C}$ is completely anisotropic, there exists a connected \'etale algebra $A\in\mathcal{D}$ such that $\mathcal{D}^0_A\simeq\mathcal{C}$ is a braided equivalence \cite[Proposition 5.15(iii)]{DMNO}.  Hence $\mathrm{FPdim}(\mathcal{C})=\mathrm{FPdim}(\mathcal{D}^0_A)\in\mathbb{Q}(\mathrm{FPdim}(\mathcal{D}))$ by Lemma \ref{lem:witt2}.
\end{proof}
 
\par Fix an algebraic number field $\mathbb{K}$.  We define
\begin{equation}\label{witteq}
\mathcal{W}_\mathbb{K}:=\{[\mathcal{C}]\in\mathcal{W}:\Phi([\mathcal{C}])\subset\mathbb{K}\}.
\end{equation}
The following are clear from the definition of $\mathcal{W}_\mathbb{K}$.

\begin{lem}\label{fieldint}
For all algebraic number fields $\mathbb{K},\mathbb{L}$,
\begin{itemize}
\item[(a)] $\mathcal{W}_\mathbb{K}\subset\mathcal{W}_\mathbb{L}$ if $\mathbb{K}\subset\mathbb{L}$,
\item[(b)] $\mathcal{W}_\mathbb{Q}\subset\mathcal{W}_\mathbb{K}$,\text{ and } 
\item[(c)] $\mathcal{W}_\mathbb{K}\cap\mathcal{W}_\mathbb{L}=\mathcal{W}_{\mathbb{K}\cap\mathbb{L}}$.
\end{itemize}
\end{lem}
Most importantly, $\mathcal{W}_\mathbb{K}$ is closed under products by Lemma \ref{lem:witt2}, proving the following.
\begin{theo}\label{wittt}
Let $\mathbb{K}$ be an algebraic number field.  The set $\mathcal{W}_\mathbb{K}:=\{[\mathcal{C}]\in\mathcal{W}:\Phi([\mathcal{C}])\subset\mathbb{K}\}$
is a Galois-invariant subgroup of $\mathcal{W}$.
\end{theo}

\begin{note}
The Galois-invariance of the subgroups $\mathcal{W}_\mathbb{K}$ is clear as inclusion in $\mathcal{W}_\mathbb{K}$ is defined only on the level of Grothendieck rings.  Hence any structure of a nondegenerately braided (or modular) fusion category on a fixed fusion ring describes elements of the same $\mathcal{W}_\mathbb{K}$.
\end{note}

\par Recall the notion of \emph{slightly degenerate} braided fusion categories, i.e. \!those braided fusion categories $\mathcal{C}$ with symmetric center braided equivalent to $\mathrm{sVec}$, the category of super vector spaces.  The analogous notions to the Witt group can be generalized \cite[Section 5]{DNO} to define $s\mathcal{W}$, the \emph{super Witt group} of slightly degenerate braided fusion categories.  We will show each super Witt equivalence class contains a representative whose dimensional grading is trivial.

\begin{theo}\label{genthm}
Let $\mathcal{C}$ be a slightly degenerate braided fusion category.  If $\mathcal{C}$ has no nontrivial Tannakian subcategories, then $\mathrm{FPdim}(X)\in\mathbb{K}_0$ for all $X\in\mathcal{C}$.
\end{theo}

\begin{proof}
There exist $s$-simple fusion subcategories $\mathcal{C}_1,\ldots,\mathcal{C}_m\subset\mathcal{C}$ determined uniquely up to a permutation of indices so $\mathcal{C}\simeq\mathcal{C}_\mathrm{pt}\boxtimes_\mathrm{sVec}\mathcal{C}_1\boxtimes_\mathrm{sVec}\cdots\boxtimes_\mathrm{sVec}\mathcal{C}_m$ \cite[Theorem 4.13(1)]{DNO}.  Also consider the fusion subcategory $\mathcal{C}_{\mathbb{K}_0}\subset\mathcal{C}$ of all $X\in\mathcal{C}$ with $\mathrm{FPdim}(X)\in\mathbb{K}_0$ (the trivial component of the dimensional grading) which, by \cite[Theorem 4.13(2)]{DNO}, has
\begin{equation}
\mathcal{D}_{\mathbb{K}_0}\simeq\mathcal{D}_\mathrm{pt}\boxtimes_\mathrm{sVec}\mathcal{D}_{i_1}\boxtimes_\mathrm{sVec}\cdots\boxtimes_\mathrm{sVec}\mathcal{D}_{i_k}
\end{equation}
for a subset $\{i_1,\ldots,i_k\}\subset\{1,\ldots,m\}$. Assume $j\in\{1,\ldots,m\}\backslash\{i_1,\ldots,i_k\}$.  For such $j$ and any $X\in\mathcal{O}(\mathcal{C}_j)$, $\mathrm{FPdim}(X)^2\in\mathbb{K}_0$ by Proposition \ref{cor:second}.  But $\mathcal{C}_j\cap\mathcal{C}_{\mathbb{K}_0}=\mathrm{sVec}$.  Hence $X\in(\mathcal{C}_j)_\mathrm{pt}\simeq\mathrm{sVec}$ or $X$ $\otimes$-generates $\mathcal{C}_j$ with $X\otimes X^\ast\in(\mathcal{C}_j)_\mathrm{pt}\simeq\mathrm{sVec}$.  The latter would imply $\mathcal{C}_j$ is a slightly degenerate (pre-modular) category of rank 3 because $\mathcal{C}_j$ is $s$-simple, which does not exist \cite[Main Theorem]{premodular}.  Moreover no such $j$ exists, and $\mathcal{C}_{\mathbb{K}_0}=\mathcal{C}$.
\end{proof}

\begin{cor}\label{corint}
Let $\mathcal{C}$ be a slightly degenerate braided fusion category.  There exists a slightly degenerate braided fusion category $\mathcal{D}$ with $[\mathcal{C}]=[\mathcal{D}]\in s\mathcal{W}$ and $\mathrm{FPdim}(X)\in\mathbb{Q}(\mathrm{FPdim}(\mathcal{D}))$ for all $X\in\mathcal{O}(\mathcal{D})$.
\end{cor}

\begin{proof}
Theorem 5.5 of \cite{DNO} states that there exists a unique (up to braided equivalence) completely $\mathrm{sVec}$-anisotropic slightly degnerate braided fusion category $\mathcal{D}$, super Witt equivalent to $\mathcal{C}$.  In other words, any connected \'etale algebra in $\mathcal{D}$ is contained in $\mathrm{sVec}$.  In particular, $\mathcal{D}$ has no nontrivial Tannakian fusion subcategories and Proposition \ref{genthm} finishes the proof.
\end{proof}

\par Proposition \ref{genthm} and the analogous Corollary \ref{corint} are false for nondegenerately braided fusion categories.  For example $\mathcal{C}(\mathfrak{sl}_2,2)$ is completely anisotropic and modular, but there exists a simple object $\lambda\in\mathcal{C}(A_1,2)$ with
\begin{equation}
\mathrm{FPdim}(\lambda)=\sqrt{2}\not\in\mathbb{Q}=\mathbb{Q}(\mathrm{FPdim}(\mathcal{C}(\mathfrak{sl}_2,2))).
\end{equation}
Symmetrically braided fusion subcategories (such as $\mathrm{sVec}$) are the only obstruction.  In general, a symmetric fusion category is \emph{super Tannakian} \cite{deligne1,deligne2}, which extends the notion of Tannakian, to include braided fusion categories which are  braided equivalent to $\mathrm{Rep}(G,z)$ where $G$ is a finite group and $z\in G$ is a central element of order 2 twisting the usual braiding.  For instance $\mathrm{sVec}$ is nothing more than $\mathrm{Rep}(\mathbb{Z}/2\mathbb{Z},z)$ where $z$ is the nontrivial central element.  In particular $z=1$ corresponds to Tannakian categories from Example \ref{tannak}.

\begin{theo}
Let $\mathcal{C}$ be a nondegenerately braided fusion category.  If $\mathcal{C}$ has no symmetric fusion subcategories, then $\mathrm{FPdim}(X)\in\mathbb{K}_0$ for all $X\in\mathcal{C}$.
\end{theo}

\begin{proof}
The slightly degenerate category $\mathcal{C}\boxtimes\mathrm{sVec}$ has no nontrivial Tannakian fusion subcategories as $\mathcal{C}_\mathrm{pt}$ is nondegenerate by assumption and the result follows from Proposition \ref{genthm}.
\end{proof}

\begin{note}
Let $\mathcal{E}$ be a symmetric fusion category.  Davydov, Nikshych, and Ostrik define the tensor product $\mathcal{C}\boxtimes_\mathcal{E}\mathcal{D}$ of nondegenerate braided fusion categories $\mathcal{C},\mathcal{D}$ over $\mathcal{E}$ \cite[Definition 4.1]{DNO}, and moreover the Witt group of nondegenerate braided fusion categories over $\mathcal{E}$, $\mathcal{W}(\mathcal{E})$ \cite[Definition 5.1]{DNO}.  These constructions preserve fields of Frobenius-Perron dimensions, and so the above results can likewise be generalized to these settings.
\end{note}

\end{section}


\begin{section}{Quantum groups at roots of unity}\label{sec:quantum}

\par A rudimentary introduction to the modular tensor categories $\mathcal{C}(\mathfrak{g},k)$ is given in Examples \ref{ex:quant1} and \ref{ex:quant2} where the reader can find further references.  Our end goal is an exact description of $\mathbb{K}_\lambda$ and $\mathbb{L}$ for arbitrary $\lambda\in\mathcal{O}(\mathcal{C}(\mathfrak{g},k))$.  Our main tools are those of Section \ref{sec:main}, and Sawin's classification of closed subsets of $\mathcal{C}(\mathfrak{g},k)$, i.e. \!a classification of fusion subrings of the Grothendieck ring of $\mathcal{C}(\mathfrak{g},k)$ \cite[Theorem 1]{Sawin06}.  Classifying fusion subrings is equivalent to classifying fusion subcategories in general \cite[Corollary 4.11.4]{tcat}.  

\par Each $\lambda\in\mathcal{O}(\mathcal{C}(\mathfrak{g},k)_\text{pt})$ has $\fpdim(\lambda)=1$ and these simple currents were classified by Fuchs \cite{fuchs1991}; with one exception ($E_8$ at level $k=2$) they correspond to extended Dynkin diagram symmetries. In particular (apart from $E_8$ at level $k=2$) the group of simple currents for $\mathcal{C}(\mathfrak{g},k)$ is isomorphic to the center $Z:=Z(G)$ of the simply connected Lie group $G$ associated to $\mathfrak{g}$.

\par If $k\neq2$, fusion subcategories of $\mathcal{C}(\mathfrak{g},k)$ are $\Delta_H\subset\mathcal{C}(\mathfrak{g},k)_\text{pt}$ (subgroups $H\subset Z$), or their relative commutants in $\mathcal{C}(\mathfrak{g},k)$ \cite[Definition 2.6]{mug1} which we denote by $\Delta_H'$.  If  $\Delta_H\subset\Delta_G$, then $\Delta_G'\subset\Delta_H'$.


\begin{subsection}{Number of dimension fields}\label{sec:int}

\par Here we argue that with one exception, for fixed $\mathfrak{g}$ and $k$, the fields $\mathbb{K}_\lambda$ for $\lambda\in\mathcal{O}(\mathcal{C}(\mathfrak{g},k))$ are at most three in number.  Let us begin with an easy warm up, before resuming the general case.

\begin{example}\label{ex:typeA}
There is only one proper fusion subcategory of $\mathcal{C}(A_1,k)$ which is not pointed when $k>2$: $\mathcal{C}(A_1,k)_\text{ad}$.  The defining (natural) representation $\Lambda_1$ $\otimes$-generates $\mathcal{C}(A_1,k)$ and $\Lambda_1\otimes\Lambda_1^\ast\in\mathcal{C}(A_1,k)_\mathrm{ad}$, so $\mathbb{K}_{\lambda}=\mathbb{Q},\mathbb{K}_0,\mathbb{K}_1$ for any $\lambda\in\mathcal{O}(\mathcal{C}(A_1,k))$ where $\mathbb{K}_1=\mathbb{K}_{\Lambda_1}$.  Thus there are four cases of when $\mathbb{Q}$, $\mathbb{K}_0$, $\mathbb{K}_1$ are equal or not.  Figure \ref{fig:appB} illustrates that all possibilities of $\mathbb{Q}\to\mathbb{K}_0\to\mathbb{K}_1$ are realized where $\to$ is inclusion of fields.
\begin{figure}[H]
\centering
\begin{tikzpicture}
\node at (-1,3) {$\mathcal{C}(A_1,1):$};
\node (a) at (0,3) {$\mathbb{Q}$};
\node (c) at (3,3) {$\mathbb{Q}$};
\node (d) at (6,3) {$\mathbb{Q}$};
\draw[->] (a) -- (c);
\draw[->] (c) -- (d);

\node at (-1,2) {$\mathcal{C}(A_1,2):$};
\node (a) at (0,2) {$\mathbb{Q}$};
\node (c) at (3,2) {$\mathbb{Q}$};
\node (d) at (6,2) {$\mathbb{Q}(\sqrt{2})$};
\draw[->] (a) -- (c);
\draw[->] (c) -- (d);

\node at (-1,1) {$\mathcal{C}(A_1,3):$};
\node (a) at (0,1) {$\mathbb{Q}$};
\node (c) at (3,1) {$\mathbb{Q}(\sqrt{5})$};
\node (d) at (6,1) {$\mathbb{Q}(\sqrt{5})$};
\draw[->] (a) -- (c);
\draw[->] (c) -- (d);

\node at (-1,0) {$\mathcal{C}(A_1,6):$};
\node (a) at (0,0) {$\mathbb{Q}$};
\node (c) at (3,0) {$\mathbb{Q}(\sqrt{2})$};
\node (d) at (6,0) {$\mathbb{Q}(\sqrt{2+\sqrt{2}})$};
\draw[->] (a) -- (c);
\draw[->] (c) -- (d);
\end{tikzpicture}
    \caption{$\fpdim$ fields of $\mathcal{C}(A_1,k)$ for $k=1,2,3,6$}%
    \label{fig:appB}%
\end{figure}
\end{example}

\begin{lem}\label{modern}
Choose any category $\mathcal{C}(\mathfrak{g},k)$.  If $k\neq2$ and $\mathcal{D}\subset\mathcal{C}(\mathfrak{g},k)$ is a fusion subcategory which is not pointed, then $\mathcal{C}(\mathfrak{g},k)_\mathrm{ad}\subset\mathcal{D}$.
\end{lem}

\begin{proof}
Since $\Delta_G\subset\Delta_Z$ for all subgroups $G\subset Z$, then $\Delta_G'\supset\Delta_Z'=\mathcal{C}(\mathfrak{g},k)_\mathrm{ad}$ \cite[Corollary 6.9]{nilgelaki} by the discussion at the beginning of this section.
\end{proof}

\begin{theo}\label{theoremlemma}
Choose any category $\mathcal{C}(\mathfrak{g},k)$ except $\mathcal{C}(E_7,2)$.  For all $\lambda\in\mathcal{O}(\mathcal{C}(\mathfrak{g},k))$, we have $\mathbb{K}_\lambda=\mathbb{Q},\mathbb{K}_0,\mathbb{K}_1$.
\end{theo}

\begin{proof}
The universal grading group of $\mathcal{C}(\mathfrak{g},k)$ is either cyclic, or Klein-4 (for Type $D_{2n}$) as these are isomorphic to the abelian groups $\mathcal{O}(\mathcal{C}(\mathfrak{g},k))_\mathrm{pt}$ \cite[Theorem 6.3]{nilgelaki}, so $\mathbb{K}_\lambda\subset\mathbb{K}_1$ for all $\lambda\in\mathcal{O}(\mathcal{C}(\mathfrak{g},k))$ by Proposition \ref{cyclic} except (potentially) in the case of Type $D_{2n}$.  But for Type $D_{2n}$, the spinor representations $\Lambda_{2n}$ and $\Lambda_{2n-1}$ have the same $\fpdim$ but lie in distinct nontrivial components of the universal grading, and thus the dimensional grading is trivial or $\mathbb{Z}/2\mathbb{Z}$ as well.

\par For all $\mathcal{C}(\mathfrak{g},k)$, $\mathbb{K}_0\subset\mathbb{K}_\lambda$ by Lemma \ref{modern} and moreover $\mathbb{K}_\lambda=\mathbb{K}_0,\mathbb{K}_1$ as $[\mathbb{K}_1:\mathbb{K}_0]=2$.  The exceptional level $k=2$ cases (aside from $\mathcal{C}(E_7,2)$) are weakly integral ($\mathcal{C}(B_n,2)$, $\mathcal{C}(D_n,2)$) or are rank 3 ($\mathcal{C}(E_8,2)$) so the statement is trivial.
\end{proof}

\begin{example}[the exception]\label{E}
Consider $\mathcal{C}(E_7,2)$ of rank 6.  We have described the Grothendieck ring $R$ explicitly in Example \ref{exe72}.  Moreso
\begin{equation}
\mathcal{C}(E_7,2)\simeq\mathcal{C}(A_1,3)_\text{ad}\boxtimes\mathcal{C}(C_2,1)
\end{equation}
is a braided equivalence.  Hence $\mathbb{K}_\lambda$ for $\lambda\in\mathcal{O}(\mathcal{C}(E_7,2))$ is $\mathbb{Q}$, $\mathbb{Q}(\sqrt{2})$, $\mathbb{Q}(\sqrt{5})$, or $\mathbb{Q}(\sqrt{2},\sqrt{5})$.
\end{example}

\begin{note}
The category $\mathcal{C}(E_7,2)$ still satisfies the hypotheses of Proposition \ref{cyclic}, i.e. \!there exists a unique quadratic extension $\mathbb{K}_1/\mathbb{K}_0$ with $\mathrm{FPdim}(x)\in\mathbb{K}_1$ for all $x\in R$.  So $\mathcal{C}(E_7,2)$ is not exceptional in this way.
\end{note}

\end{subsection}


\begin{subsection}{Explicit descriptions of dimension fields}\label{explicitdim}

\par Proposition \ref{theoremlemma} allows each number field $\mathbb{K}_\lambda$ for $\lambda\in\mathcal{O}(\mathcal{C}(\mathfrak{g},k))$ to be described explicitly (Figure \ref{fig:B}) as this only depends on the component of $\lambda$ in the dimensional grading (see Note \ref{note1}).  To compute these two fields we are free to choose test weights from the trivial and nontrivial components of the dimensional grading.  Explicitly, if $\lambda=\sum_j\ell_j\Lambda_j$ is an integral sum of fundamental weights, the dimensional grading of $\lambda$ is the parity of $\sum_jj\ell_j$ for Types $A_n,C_n$, $\ell_n$ for Type $B_n$, $\ell_{n-1}+\ell_n$ for Type $D_n$, and $\ell_4+\ell_6+\ell_7$ for Type $E_7$ (the other types have trivial universal grading group).

\par If $n,m\in\mathbb{Z}_{\geq1}$, denote $[n]_m:=\sin(n\pi/m)/\sin(\pi/m)$.  Note that $\cos(nx)=T_n(\cos(x))$ where $T_n$ (the $n^\text{\tiny th}$ Chebyshev polynomial) is even (respectively, odd) if and only if $n$ is even (respectively, odd). This means that any nonzero finite product $\prod_i\cos(a_i\pi/m)$ where $a_i,m$ are integers, lies in $\mathbb{Q}_m$ if $\sum_ia_i$ is even or $m$ is odd, otherwise it lies in $\mathbb{Q}_{2m}$ and not $\mathbb{Q}_m$. 

\begin{lem}\label{lem:a1}
If $n,m\in\mathbb{Z}$ such that $0<n<m$, then
\begin{equation*}
\mathbb{Q}([n]_m)=\left\{\begin{array}{lr} \mathbb{Q} & \textnormal{if }n=1\textnormal{ or }n=m-1 \\
\mathbb{Q}_{2m} & \textnormal{if }n\textnormal{ is even, and }1<n<m-1 \\
\mathbb{Q}_{m} & \textnormal{otherwise} \end{array}\right..
\end{equation*}
\end{lem}

\begin{proof}
When $m=2$ the statement is clear so let $m\geq3$.  Then $[n]_m$ is the $\fpdim$ of $(n-1)\Lambda_1\in\mathcal{O}(\mathcal{C}(A_1,m-2))$.  Proposition \ref{theoremlemma} implies that $\mathbb{K}_{(n-1)\Lambda_1}$ is either $\mathbb{K}_{\Lambda_1}=\mathbb{Q}([2]_m)$ if $n$ is even and $(n-1)\Lambda_1$ is not a simple current (if and only if $n=1$ or $n=m-1$), or $\mathbb{K}_{2\Lambda_1}=\mathbb{Q}([3]_m)$ if $n$ is odd and $(n-1)\Lambda_1$ is not a simple current. But $\mathbb{Q}([2]_m)=\mathbb{Q}(2\cos(\pi/m))=\mathbb{Q}_{2m}$ and $\mathbb{Q}([3]_m)=\mathbb{Q}(1+2\cos(2\pi/m))=\mathbb{Q}_m$.
\end{proof}

\par\textbf{Type $A_n$, $n\in\mathbb{Z}_{\geq1}$: } Our test weights will be the adjoint representation $\Lambda_1+\Lambda_n$ in the trivial component and the defining (natural) representation $\Lambda_1$ when $k>1$.  The category $\mathcal{C}(A_n,1)$ is pointed, otherwise using (\ref{weyl}), $\fpdim(\Lambda_1)=[n+1]_\kappa$ and $\fpdim(\Lambda_1+\Lambda_2)=[n+1]_\kappa^2-1$.  Therefore $\mathbb{K}_1=\mathbb{Q}_\kappa$ if $n$ is even and $\mathbb{K}_1=\mathbb{Q}_{2\kappa}$ if $n$ is odd.  But if $n$ is odd, $\mathbb{Q}_\kappa=\mathbb{Q}_{2\kappa}$ when $k$ is odd.   For $n=1$ we then have $\mathbb{K}_0=\mathbb{Q}_\kappa$ as
\begin{equation}
\mathrm{FPdim}(\Lambda_1+\Lambda_n)=4\cos(\pi/\kappa)^2-1=2\cos(2\pi/\kappa)+1.
\end{equation}
In general, $\fpdim(\Lambda_1+\Lambda_n)$ is one less than the $\fpdim$ of an object in $\mathcal{C}(\mathfrak{sl}_2,\kappa-2)_\mathrm{ad}$ hence $\mathbb{K}_0=\mathbb{Q}_\kappa$.  In sum, all fields are $\mathbb{Q}_{\kappa}$ except $\mathbb{K}_1=\mathbb{Q}_{2\kappa}$ when $n$ is odd and $k$ is even.


\par\textbf{Type $B_n$, $n\in\mathbb{Z}_{\geq3}$: } The categories $\mathcal{C}(B_n,1)$ are Ising categories for which $\mathbb{K}_0=\mathbb{Q}$ and $\mathbb{K}_1=\mathbb{Q}(\sqrt{2})$ and the categories $\mathcal{C}(B_n,2)$ are metaplectic where $\mathbb{K}_0=\mathbb{Q}$ and $\mathbb{K}_1=\mathbb{Q}(\sqrt{2n+1})$.  Assume $k>2$.  Our test weights will be the defining (natural) representation $\Lambda_1$ in the trivial component and the spinor representation $\Lambda_n$.  We compute $\fpdim(\Lambda_1)=[2n]_{\kappa}+1$ hence $K_0=\mathbb{Q}_{2\kappa}$ by Lemma \ref{lem:a1}, and
\begin{equation}
\fpdim(\Lambda_n)=2^n\prod_{i=1}^n\cos((2i-1)\pi/(2\kappa)).
\end{equation}
By the discussion above, $\mathbb{K}_1=\mathbb{Q}_{2\kappa}$ if $n$ is even and $k>2$ and $\mathbb{K}_1=\mathbb{Q}_{4\kappa}$ if $n$ is odd and $k>2$.


\par\textbf{Type $C_n$, $n\in\mathbb{Z}_{\geq2}$: }  If $k=1$ the dimensions coincide with those of $\mathcal{C}(A_1,n)$ so assume $k>1$.  Our test weights will be the defining (natural) representation $\Lambda_1$ in the nontrivial component and the weights $2\Lambda_1$ and $\Lambda_2$ in the trivial component for ease of argument.  We compute $\fpdim(\Lambda_1)=[2n+1]_{2\kappa}-1$ hence $\mathbb{K}_1=\mathbb{Q}_{2\kappa}$ by Lemma \ref{lem:a1}, and
\begin{equation}\fpdim(2\Lambda_1)=([n+1]_\kappa-1)[2n+1]_{2\kappa}.
\end{equation}
We know $[2n+1]_{2\kappa}$ generates $\mathbb{Q}_{2\kappa}$ and when $n$ is even $[n+1]_\kappa$ lies in $\mathbb{Q}_\kappa$. So when $n$ is even, $\mathbb{K}_0=\mathbb{Q}_{2\kappa}$.  Lastly, if $n$ is odd,
\begin{equation}
\fpdim(\Lambda_2)=([n]_\kappa-1)[2n+1]_{2\kappa}
\end{equation}
and $[n]_\kappa\in\mathbb{Q}_\kappa$ so $\fpdim(\Lambda_2)$ generates $\mathbb{Q}_{2\kappa}$. Thus in all cases $\mathbb{K}_0=\mathbb{Q}_{2\kappa}$.  
%


\par\textbf{Type $D_n$, $n\in\mathbb{Z}_{\geq4}$: } The categories $\mathcal{C}(D_n,1)$ are pointed so $\mathbb{K}_0=\mathbb{K}_1=\mathbb{Q}$ and the categories $\mathcal{C}(D_n,2)$ are metaplectic with $\mathbb{K}_0=\mathbb{Q}$ and $\mathbb{K}_1=\mathbb{Q}(\sqrt{n})$.  Assume $k>2$.  Our test weights will be the defining (natural) representation $\Lambda_1$ in the trivial component and the spinor representation $\Lambda_n$.  We find $\fpdim(\Lambda_1)=[2n-1]_{\kappa}+1$ hence $K_0=\mathbb{Q}_{\kappa}$ by Lemma \ref{lem:a1}, and
\begin{equation}
\mathrm{FPdim}(\Lambda_n)=2^{n-1}\prod_{j=1}^{n-1}\cos(j\pi/\kappa).
\end{equation}
By the discussion at the beginning of the section, this implies (with $k>2$) $\mathbb{K}_1=\mathbb{Q}_\kappa$ when $n\equiv0,1\pmod{4}$ and $\mathbb{K}_1=\mathbb{Q}_{2\kappa}$ otherwise.  But $\mathbb{Q}_\kappa=\mathbb{Q}_{2\kappa}$ when $\kappa$ is odd, thus $\mathbb{K}_0=\mathbb{K}_1$ in all cases except when $n\equiv2,3\pmod{4}$ and $\kappa$ is even.

%

\par\textbf{Type $G_2$:}  The dimensional grading is trivial so we need only one test weight.  Let $\lambda$ be the 7-dimensional irreducible representation of $\mathfrak{g}_2$, with
\begin{equation}
\fpdim(\lambda)=1+2\cos(2\pi/3\kappa)+2\cos(8\pi/3\kappa)+2\cos(10\pi/3\kappa).
\end{equation}
Hence $\mathbb{K}_0=\mathbb{K}_1=\mathbb{Q}(\cos(2\pi/3\kappa)+\cos(8\pi/3\kappa)+\cos(10\pi/3\kappa))$.  This manifestly lies in $\mathbb{Q}_{3\kappa}$.  It lies in (and hence generates) a proper subfield, if and only if there exists an integer $1<\ell<3\kappa-1$ coprime to $3\kappa$ such that
\begin{equation}
\cos\left(\frac{2\pi\ell}{3\kappa}\right)+\cos\left(\frac{8\pi\ell}{3\kappa}\right)+\cos\left(\frac{10\pi\ell}{3\kappa}\right)=\cos\left(\frac{2\pi}{3\kappa}\right)+\cos\left(\frac{2\pi}{3\kappa}\right)+\cos\left(\frac{2\pi}{3\kappa}\right).
\end{equation}
But for any $\ell$, $\cos(2\pi\ell/3\kappa)<\cos(2\pi/3\kappa)$, so either $4\ell\equiv\pm1,\pm2,\pm3\pmod{3\kappa}$ or $5\ell\equiv\pm 1,\pm 2,\pm 3,\pm4\pmod{3\kappa}$.  This same Galois automorphism $\sigma_\ell\in\mathrm{Gal}(\mathbb{Q}_{3\kappa}/\mathbb{Q})$ would have to fix $s_{\mathbbm{1},\mathbbm{1}}^{-2}=\mathrm{FPdim}(\mathcal{C}(\mathfrak{g}_2,k))\in\mathbb{K}_0=\mathbb{K}_1$, so $s_{\hat{\sigma}_\ell(\mathbbm{1}),\mathbbm{1}}=s_{\mathbbm{1},\mathbbm{1}}$ and hence $\hat{\sigma}_\ell(\mathbbm{1})$ is a simple current. But the only simple current of $\mathcal{C}(\mathfrak{g}_2,k)$ is $\mathbbm{1}$, so we must have $\hat{\sigma_\ell}(\mathbbm{1})=\mathbbm{1}$. This requires $t_{\mathbbm{1},\mathbbm{1}}^{\ell^2}=t_{\mathbbm{1},\mathbbm{1}}$, i.e. $\ell^2\langle\rho,\rho\rangle\equiv\langle\rho,\rho\rangle\pmod{2\kappa}$, hence $7\ell^2\equiv 7\pmod{3\kappa}$. Given the prime decomposition $3\kappa=\prod_pp^{n_p}$, this becomes $\ell\equiv\pm1\pmod{p^{n_p}}$ for all $p$ except $p=2,7$ where we require $\ell\equiv\pm 1\pmod{p^{n_p-1}}$. The signs can depend on $p$.

\par Putting these together, we achieve a small list of prospective levels (14 to be exact) to examine.  For example, no prime $p>7$ can divide $\kappa$, since it would have to satisfy $4\equiv \pm 1,\pm 2,\pm 3$ or $5\equiv \pm 1,\pm 2,\pm 3,\pm 4$  (mod $p$). Likewise, if 9 or 5 divide $3\kappa$, then $\kappa$ must be odd.  Thus we have exceptions $k=1$ where $\mathbb{K}_0=\mathbb{K}_1=\mathbb{Q}_{\kappa}$ and level $k=3$ where $\mathbb{K}_0=\mathbb{K}_1=\mathbb{Q}(\sqrt{21})$.


\par\textbf{Type $F_4$:} The dimensional grading is trivial so we need only one test weight.  Let $\lambda$ be the 26-dimensional irreducible representation of $F_4$ with
\begin{equation*}
\fpdim(\lambda)=[12]_\kappa+[7]_\kappa+2\cos(10\pi/\kappa)+2\cos(5\pi/\kappa)+2\cos(\pi/\kappa)+1
\end{equation*}
This manifestly lies in $\mathbb{Q}_{2\kappa}$. We will use the same argument as in the $G_2$ case to prove that $\fpdim(\lambda)$ generates $\mathbb{Q}_{2\kappa}$.  As long as $\kappa>12$, any $1<\ell<2\kappa-1$ coprime to $2\kappa$ has $\sigma_\ell([12]_\kappa)<[12]_\kappa$ and $\sigma_\ell([7]_\kappa)\leq [7]_\kappa$, since for $\kappa>12$ these are $\mathrm{FPdims}$ of simple objects in $\mathcal{C}(A_1,\kappa-2)$, so we require $5\ell\equiv\pm1,\pm2,\pm3,\pm4$ (mod $\kappa$). But we would require $t_{\mathbbm{1},\mathbbm{1}}^{\ell^2}=t_{\mathbbm{1},\mathbbm{1}}$ for the same reason as for Type $G_2$, which reduces to $\ell^2\equiv 39\pmod{2\kappa}$. Again these two constraints eliminate all but a small number of levels which have to be dealt with explicitly.  These exceptions are at level $k=1$ with $\mathbb{K}_0=\mathbb{K}_1=\mathbb{Q}(\sqrt{5})$, at level $k=3$ with $\mathbb{K}_0=\mathbb{K}_1=\mathbb{Q}(\sqrt{6})$ and at level $k=4$ with $\mathbb{K}_0=\mathbb{K}_1=\mathbb{Q}(\cos(2\pi/13)-\cos(3\pi/13))$.


\par\textbf{Type $E_6$:} The dimensional grading is trivial so we need only one test weight.  Let $\lambda$ be a 27-dimensional irreducible representation of $E_6$ with
\begin{equation}
\fpdim(\lambda)=[17]_{\kappa}+[9]_{\kappa}+1.
\end{equation}
Hence $\mathbb{K}_0=\mathbb{K}_1=\mathbb{Q}( [17]_{\kappa}+[9]_{\kappa})\in\mathbb{Q}_\kappa$.  For $\kappa>17$, $[17]_\kappa+[9]_\kappa$ is the $\fpdim$ of an object in $\mathcal{C}(A_1,\kappa-2)_\mathrm{ad}$ so $\mathbb{K}_0=\mathbb{K}_1=\mathbb{Q}_\kappa$.  It suffices to consider $12<\kappa\le 17$ individually.  This produces the exceptions $k=1$ with $\mathbb{K}_0=\mathbb{K}_1=\mathbb{Q}$ and $k=3$ with $\mathbb{K}_0=\mathbb{K}_1=\mathbb{Q}(\sqrt{5})$.


\par\textbf{Type $E_7$:} Our test weights will be the 56-dimensional irreducible representation $\lambda$ in the nontrivial component of the dimensional grading with $\fpdim(\lambda)=[28]_{\kappa}+[18]_{\kappa}+[10]_{\kappa}$ and the adjoint representation $\theta$, with
\begin{equation}
\fpdim(\theta)=[35]_{\kappa}+[27]_{\kappa}+[23]_{\kappa}+[19]_{\kappa}+[15]_{\kappa}+[11]_{\kappa}+[3]_{\kappa}.
\end{equation}  These manifestly live in $\mathbb{Q}_{2\kappa}$ and $\mathbb{Q}_{\kappa}$ respectively, and by the above arguments we know they will generate those   fields for $\kappa>28$ and $\kappa>35$ respectively. The values $18<\kappa\le 35$ can be dealt with separately giving exceptional levels $k=1$ with $\mathbb{K}_0=\mathbb{K}_1=\mathbb{Q}$, $k=2$ with $\mathbb{K}_0=\mathbb{K}_1=\mathbb{Q}(\sqrt{2},\sqrt{5})$ and  $\mathbb{K}_0=\mathbb{K}_1=\mathbb{Q}(\sqrt{5})$, and $k=3$ with $\mathbb{K}_0=\mathbb{K}_1=\mathbb{Q}(\sqrt{21})$.


\par\textbf{Type $E_8$:}  The dimensional grading is trivial so we need only one test weight.  Let $\lambda$ be the adjoint representation of $E_8$ with
\begin{equation}
\fpdim(\lambda)=[59]_{\kappa}+[47]_{\kappa}+[39]_{\kappa}+[35]_{\kappa}+[27]_{\kappa}+[23]_{\kappa}+[15]_{\kappa}+[3]_{\kappa}.
\end{equation}
This manifestly live in $\mathbb{Q}_{\kappa}$ .  By the reasoning above, it suffices to consider $30<\kappa\le 59$ giving exceptions at levels $k=1$ where $\mathbb{K}_0=\mathbb{K}_1=\mathbb{Q}$, $k=2$ where $\mathbb{K}_0=\mathbb{K}_1=\mathbb{Q}(\sqrt{2})$, $k=3$ where $\mathbb{K}_0=\mathbb{K}_1=\mathbb{Q}_{11}$, and $k=5$ where $\mathbb{K}_0=\mathbb{K}_1=\mathbb{Q}(\cos(2\pi/35)+\cos(12\pi/35))$ (a degree 6 extension of $\mathbb{Q}$).  Otherwise $\mathbb{K}_0=\mathbb{K}_1=\mathbb{Q}_\kappa$.

\end{subsection}


\begin{subsection}{Categories with dimension fields of small degree}\label{secsecint}

\par Here we justify the categories $\mathcal{C}(\mathfrak{g},k)$ for which $[\mathbb{K}_0:\mathbb{Q}]$ is small in Figure \ref{fig:A}.  The curious reader can extend these results to field extensions of any degree with the observation that for all $N\geq3$, $\mathbb{Q}_N$ is a degree $\phi(N)/2$ extension of $\mathbb{Q}$ where $\phi(N)$ is the Euler totient function.  In particular $[\mathbb{Q}_n:\mathbb{Q}]\leq9$ implies $n\leq54$ and $\varphi(n)/2\leq9$.  The low-level exceptions for the classical Lie algebras are well-studied and for the finitely-many exceptional levels for the exceptional Lie algebras the $\fpdim$ fields can be checked by hand.  So our argument can be restricted to the nonexceptional levels in Figure \ref{fig:B}.

\par For Type $A_n$, $\mathbb{K}_0=\mathbb{Q}_\kappa$ thus all $\kappa$ are subject to $\kappa=k+n+1\leq54$, thus $n\leq52$ and $k\leq53-n$, yielding finitely-many non-exceptional cases for analysis.  For Type $B_n$, $\mathbb{K}_0=\mathbb{Q}_{2\kappa}$, thus we require $\kappa=k+2n-1\leq27$.  Moreover $2n\leq26$ and $k\leq28-2n$, with a finite set of acceptable values.  For Type $C_n$, $\mathbb{K}_0=\mathbb{Q}_{2\kappa}$, thus $\kappa=k+n+1\leq27$, and hence $n\leq25$ and $k\leq28-n$.  For Type $D_n$, $\mathbb{K}_0=\mathbb{Q}_{\kappa}$, thus we require $k+2n-2\leq54$ .  For Type $G_2$, $\mathbb{K}_0=\mathbb{Q}_{3\kappa}$, so we require $\kappa/3=k+4\leq18$ thus $k\leq14$.  For Type $F_4$, $\mathbb{K}_0=\mathbb{Q}_{2\kappa}$, so we require $\kappa/2=k+9\leq27$ thus $k\leq18$.  For Type $E_6$, $\mathbb{K}_0=\mathbb{Q}_{\kappa}$ so we require $\kappa=k+12\leq54$ thus $k\leq42$.  For Type $E_7$, $\mathbb{K}_0=\mathbb{Q}_\kappa$ so we require $\kappa=k+18\leq54$ thus $k\leq36$.  And lastly for Type $E_8$, $\mathbb{K}_0=\mathbb{Q}_\kappa$ so we require $\kappa=k+30\leq54$ thus $k\leq24$.

\end{subsection}


\begin{subsection}{The field of all Verlinde eigenvalues}

Given any modular tensor category $\mathcal{C}$, the number field $\mathbb{L}$ generated by all Verlinde eigenvalues $s_{X,Y}/s_{\mathbbm{1},Y}$ (or equivalently by all fields $\mathbb{L}_X$ in Equation (\ref{eqx})) is of interest because it constrains the $T$-matrix.  In particular define
\begin{equation}
\mathbb{L}:=\mathbb{Q}\left(\frac{s_{X,Y}}{s_{\mathbbm{1},Y}}:X,Y\in\mathcal{O}(\mathcal{C})\right).
\end{equation}
We have already seen in Corollary \ref{bettercor} how $\mathbb{L}$, through its subfields $\mathbb{L}_X$, gives an upper bound on the orders of $t_X$. But it also gives lower bounds. Note that $\mathbb{K}_1\subset\mathbb{L}$; the equation
\begin{equation}
s_{X,Y}=\overline{t_\mathbbm{1}}t_Xt_Y\sum_{Z\in\mathcal{O}(\mathcal{C})}c_{XY}^Z\overline{t_Z}s_{Z,\mathbbm{1}}
\end{equation}
(see e.g. \cite[Equation (2.35)]{fuchs1994fusion}) says that $\mathbb{L}\subset \mathbb{K}_\mathbbm{1}(t_Z\overline{t_\mathbbm{1}}:Z \in\mathcal{O}(\mathcal{C}))$. Moreover, by Galois considerations, the rank $r$ of $\mathcal{C}$ must be large enough that the group $\mathrm{Gal}(\mathbb{L}/\mathbb{Q})$ is isomorphic to a subgroup of the symmetric group on $r$ elements.

\begin{theo}\label{49}
For all $\mathcal{C}(\mathfrak{g},k)$, $\mathbb{L}=\mathbb{K}_1$ except for the low level exceptions in Figure \ref{fig:B}, as well as Type $A_n$ for $n>1$, Type $D_n$ for $n$ odd, and Type $E_6$ for $k>1$.
\end{theo}

The field $\mathbb{L}$ for $\mathcal{C}(A_n,k)$ when $n>1$ and $k>2$, equals $\mathbb{Q}(\xi_{(n+1)\kappa})$. For $\mathcal{C}(D_n,k)$ when $n$ is odd and $k>2$,  $\mathbb{L}=\mathbb{Q}_{4\kappa}$. For $\mathcal{C}(E_6,k)$ when $k>1$,  $\mathbb{L}=\mathbb{Q}_{3\kappa}$. The computation of $\mathbb{L}$ for type $A_n$ was first done in \cite[Corollary 3]{Gannon1999}; the argument we give below is new. The remainder of this subsection is devoted to the proof of Proposition \ref{49}.

\par For all $\lambda,\mu\in\mathcal{O}(\mathcal{C}(\mathfrak{g},k))$ where $\mathfrak{g}$ and $k$ are arbitrary, $s_{\lambda,\mu}/s_{\mathbbm{1},\mu}$ is the $\mathfrak{g}$-character for $\lambda$, evaluated at $-2\pi i(\mu+\rho)/\kappa$. One consequence of this is that $s_{\lambda,\mu}/s_{\mathbbm{1},\mu}$ can be expressed as a polynomial in $s_{\Lambda_i,\mu}/s_{\mathbbm{1},\mu}$ where $\Lambda_i$ are the fundamental weights. Another consequence uses the Weyl character formula to obtain
\begin{equation}
\dfrac{s_{\lambda,\mu}}{s_{\mathbbm{1},\mu}}=\frac{\sum_{w\in W}\det(w)\exp(-\langle \lambda+\rho,w(\mu+\rho)\rangle/\kappa)}{\sum_{w\in W}\det(w)\exp(-\langle \lambda+\rho,w(\rho)\rangle/\kappa)}\label{kac}
\end{equation}  
where $W$ is the (finite) Weyl group of $\mathfrak{g}$. The Weyl denominator identity recovers Equation (\ref{weyl}).

\begin{lem}\label{claim0}
Let $\mathcal{C}$ be a modular tensor category.  If $X\neq Y\in\mathcal{O}(\mathcal{C})$, then there exists $Z\in\mathcal{O}(\mathcal{C})$ such that $s_{X,Z}\neq s_{Y,Z}$.
\end{lem}

\begin{proof}
We will prove the contrapositive of this statement, so assume $s_{X,Z}= s_{Y,Z}$ for all $Z\in\mathcal{O}(\mathcal{C})$. Then by the unitarity of $S$ \cite[Proposition 2.12]{ENO},
\begin{equation}
1= \sum_{Z\in\mathcal{O}(\mathcal{C})}\overline{s_{X,Z}}s_{X,Z}=\sum_{Z\in\mathcal{O}(\mathcal{C})} \overline{s_{X,Z}}s_{Y,Z}=\delta_{X,Y},
\end{equation}
and moreover $X=Y$.
\end{proof}

\begin{paragraph}{Type $A_n$.} When $k=1$, $\mathbb{K}_0=\mathbb{K}_1=\mathbb{Q}$, so consider $k>1$.  For $\lambda=\sum_{j=1}^nl_j\Lambda_j$, write $t(\lambda):=\sum_jjl_j$.  We have proven (Section \ref{explicitdim}) that $\mathbb{Q}_\kappa\subset\mathbb{K}_1\subset\mathbb{L}$.  From (\ref{kac}), $\mathbb{L}\subset\mathbb{Q}(\xi_{(n+1)\kappa})$. We will prove equality when $n>1$ and $k>2$.

\begin{lem}\label{claim1}
For all $\lambda,\mu\in\mathcal{O}(\mathcal{C}(\mathfrak{sl}_n,k))$,
\begin{equation}
\dfrac{s_{\lambda,\mu}}{s_{\mathbbm{1},\mu}}\in\exp\left(2\pi i\dfrac{t(\lambda)t(\mu+\rho)}{(n+1)\kappa}\right)\mathbb{Q}(\xi_\kappa).
\end{equation}
\end{lem}

\begin{proof}
When $1\leq i\leq j\leq n$, $\langle\Lambda_i,\Lambda_j\rangle=\frac{i(n+1-j)}{n+1}$ so $\langle\Lambda_i,\Lambda_j\rangle\equiv-\frac{ij}{n+1}\pmod{1}$ for all $1\leq i,j\leq n$. Thus for arbitrary weights $\lambda=\sum_{i=1}^na_i\Lambda_i$ and $\mu=\sum_{j=1}^nb_j\Lambda_j$,
\begin{equation}
\langle\lambda,\mu\rangle=\frac{-1}{n+1}\sum_{1\leq i,j\leq n}a_ib_jij=\frac{-1}{n+1}t(\lambda)\,t(\mu).
\end{equation}
Now, as indicated above, $s_{\lambda,\mu}/s_{\mathbbm{1},\mu}$ is the character of $\lambda$ evaluated at $-2\pi i(\mu+\rho)/\kappa$ so it is a sum of terms of the form $\exp(-2\pi i\langle\lambda-\alpha,\mu+\rho\rangle/\kappa)$ for vectors $\alpha$ in the root lattice. But $\langle\alpha,\mu+\rho\rangle\in\mathbb{Z}$ so $\langle\lambda-\alpha,\mu+\rho\rangle\equiv\frac{-1}{n+1}t(\lambda)\,t(\mu+\rho)\pmod{1}$ which proves our claim.
\end{proof}

We want to show $\mathbb{L}$ contains $\mathbb{Q}(\xi_\kappa)$. We know $\mathbb{L}$ contains $\mathbb{K}_1$ and hence $\mathbb{Q}_\kappa$.  If we take $\lambda=2\Lambda_1+\Lambda_{n-1}$ (possible since $k\geq3$), then $\lambda$ is not self-dual since $n>1$. Hence by Lemma \ref{claim0}, there exists $\mu$ such that $s_{\lambda,\mu}\neq s_{\lambda^\ast,\mu}=\overline{s_{\lambda,\mu}}$. This means $s_{\lambda,\mu}/s_{\mathbbm{1},\mu}$ is nonreal, and by Lemma \ref{claim1} lies in $\mathbb{Q}(\xi_\kappa)$, so we find that $\mathbb{L}$ contains $\mathbb{Q}(\xi_\kappa)$.

Moreover, $\delta\otimes\Lambda_1\not\cong \Lambda_1$ for any simple current $\delta$, so by Lemma \ref{claim0} there exists $\mu$ with $t(\mu)$ coprime to $n+1$, and $s_{\Lambda_1,\mu}\neq 0$. By Lemma \ref{claim1}, this means $\xi_{(n+1)\kappa}^{t(\mu+\rho)}\in\mathbb{L}$. But again by Lemma \ref{claim1} using $s_{\Lambda_1,\mathbbm{1}}/s_{\mathbbm{1},\mathbbm{1}}$, $\xi_{(n+1)\kappa}^{t(\rho)}\in\mathbb{L}$. Hence $\xi_{(n+1)\kappa}^{t(\mu)}\in\mathbb{L}$.

\end{paragraph}

\begin{paragraph}{Type $B_n$.} We will show that $\mathbb{L}=\mathbb{K}_1$. By (\ref{kac}), $\mathbb{L}$ is contained in $\mathbb{Q}(\xi_{4\kappa})$. It is also real, since duality is trivial, so in fact $\mathbb{L}$ is contained in $\mathbb{Q}_{4\kappa}$. But $\mathbb{L}$ must contain $\mathbb{K}_1$. Thus $\mathbb{L}=\mathbb{Q}_{4\kappa}$ except possibly when $n$ is even, when it contains $\mathbb{Q}_{2\kappa}$. If $n$ is even,
\begin{equation}
\dfrac{s_{\Lambda_n,\mu}}{s_{\mathbbm{1},\mu}}=2^n\prod_{j=1}^n\cos\left(\dfrac{2\pi (\mu[j]+n-j+1/2)}{2\kappa}\right)
\end{equation}
where $\mu[j]=\sum_{i=j}^{n-1}a_i+a_n/2$ assuming $\mu=\sum_{i=1}^na_i\Lambda_i$. We see that either all $(\mu[j]+n-j+1/2)/(2\kappa)$ lie in $\frac{1}{4\kappa}+\frac{1}{2\kappa}\mathbb{Z}$ (hence $s_{\Lambda_n,\mu}/s_{\mathbbm{1},\mu}\in\mathbb{Q}_{2\kappa}$ since $n$ is even), or all lie in  $\frac{1}{2\kappa}\mathbb{Z}$ (hence $s_{\Lambda_n,\mu}/s_{\mathbbm{1},\mu}\in\mathbb{Q}_{2\kappa}$). Now, it is manifest from (\ref{kac}) that $s_{\lambda,\mu}/s_{\mathbbm{1},\mu}\in\mathbb{Q}_{2\kappa}$ when $\lambda$ is a nonspinor, e.g. for $\lambda=\Lambda_i$ for all $i<n$.

\end{paragraph}

 \begin{paragraph}{Type $C_n$.}  We will show for $k>1$ that $\mathbb{L}=\mathbb{K}_1$. By (\ref{kac}), $\mathbb{L}$ is contained in $\mathbb{Q}(\xi_{2\kappa})$. It is also real, since charge conjugation is trivial, so in fact $\mathbb{L}$ is contained in $\mathbb{Q}_{2\kappa}$.  But we also know $\mathbb{L}$ must contain $\mathbb{K}_1$. Thus $\mathbb{L}=\mathbb{Q}_{2\kappa}$.
 
\end{paragraph}

\begin{paragraph}{Type $D_n$.}  Assume $k>2$.  When $n$ is odd, $\mathbb{L}$ must contain $\mathbb{Q}_\kappa$, and be contained in $\mathbb{Q}(\xi_{4\kappa})$. We see that
\begin{equation}
\dfrac{s_{\Lambda_n,\mu}}{s_{\mathbbm{1},\mu}}+\dfrac{s_{\Lambda_{n-1},\mu}}{s_{\mathbbm{1},\mu}}= 2^n\prod_{j=1}^n\cos\left(2\pi \frac{\mu[j]+n-j}{2\kappa}\right)
\end{equation}
where $\mu[j]=\sum_{i=j}^{n-1}b_j+(b_n-b_{n-1})/2$ assuming $\mu=\sum_{i=j}^nb_j\Lambda_j$.  Hence taking $\mu=\Lambda_n$ we see that it must lie in $\mathbb{Q}_{4\kappa}$ (and not $\mathbb{Q}_{2\kappa}$).  Lemma \ref{claim0} applied to  $\Lambda_n^\ast\neq \Lambda_n$ then means $\mathbb{L}=\mathbb{Q}(\xi_{4\kappa})$.
   
When $n$ is even, $\langle\Lambda_i,\Lambda_j\rangle\in\frac{1}{2}\mathbb{Z}$ which means $\langle\lambda,\mu\rangle\in\frac{1}{2}\mathbb{Z}$. Also duality is trivial, so $\mathbb{L}\subseteq\mathbb{Q}_{2\kappa}$. Also, $\mathbb{K}_1\supseteq\mathbb{Q}_\kappa$ lies in $\mathbb{L}$. To show $\mathbb{L}=\mathbb{Q}_{2\kappa}$, note that
\begin{equation}
\dfrac{s_{\Lambda_1,\Lambda_n}}{s_{\mathbbm{1},\Lambda_n}}=2\sum_{j=1}^n\cos\left(\frac{2\pi}{\kappa}\left(r-j+\frac{1}{2}\right)\right)
\end{equation}
This must lie in $\mathbb{Q}_{2\kappa}$, but when $k$ is even, not $\mathbb{Q}_\kappa$: use the Galois automorphism of $1+\kappa$.
   
\end{paragraph}

\begin{paragraph}{Type $E_6$.}  One may easily verify $\mathbb{L}=\mathbb{Q}(\xi_3)$ for $\mathcal{C}(E_6,1)$, so assume $k>1$.  It is clear $\mathbb{Q}(\xi_{3\kappa})\supseteq\mathbb{L}\supseteq\mathbb{K}_1=\mathbb{Q}_\kappa$. We compute $\langle\lambda,\mu\rangle\equiv\frac{1}{3}t(\lambda)t(\mu)$ where $t(\lambda)=\lambda_1-\lambda_2+\lambda_4-\lambda_5$.  So $s_{\lambda,\mu}/s_{\mathbbm{1},\mu}\in\xi_{3\kappa}^{-t(\lambda)t(\mu)}\mathbb{Q}(\xi_\kappa)$. Now $(\Lambda_1+\Lambda_2)^\ast=\Lambda_4+\Lambda_5\not\cong\Lambda_1+\Lambda_2$, so $s_{\Lambda_1+\Lambda_2,\mu}\not\in\mathbb{R}$ for some $\mu$ by Lemma \ref{claim0}, so $\mathbb{L}$ contains all of $\mathbb{Q}(\xi_\kappa)$. Also, $\Lambda_1\not\cong \delta\otimes\Lambda_1$ where $\delta$ is a nontrivial simple current, so by Lemma \ref{claim0} $s_{\Lambda_1,\mu}\neq 0$ for some $\mu$ with $t(\mu)$ coprime to 3, so in fact $\mathbb{L}=\mathbb{Q}(\xi_{3\kappa})$ as desired.

\end{paragraph}

\begin{paragraph}{Type $E_7$ and $E_8$.} One may verify $\mathbb{L}=\mathbb{K}_1=\mathbb{Q}_{2\kappa}$ for $\mathcal{C}(E_7,k)$ and $\mathbb{L}=\mathbb{K}_1=\mathbb{Q}_\kappa$ for $\mathcal{C}(E_8,k)$ by the usual arguments.
\end{paragraph}

\begin{paragraph}{Type $F_4$.} We compute $\langle\lambda,\mu\rangle\in\frac{1}{2}\mathbb{Z}$ and duality is trivial, so manifestly $\mathbb{L}\subseteq\mathbb{Q}_{2\kappa}$. But (apart from $k=1,3,4$) $\mathbb{K}_1=\mathbb{Q}_{2\kappa}$ so we are done except for those 3 levels.

\end{paragraph}

\begin{paragraph}{Type $G_2$.} We compute $\langle\lambda,\mu\rangle\in\frac{1}{3}\mathbb{Z}$ and duality is trivial, so manifestly $\mathbb{L}\subseteq\mathbb{Q}_{3\kappa}$. But (apart from $k=1,3$) $\mathbb{K}_1=\mathbb{Q}_{3\kappa}$ so done except for those 2 levels.

\end{paragraph}

\end{subsection}

\end{section}


\begin{section}{Discussion}\label{discussion}

In many ways, the categories $\mathcal{C}(\mathfrak{g},k)$ and their relatives are the most complex examples of fusion, braided, and modular categories in the current literature.  But here we have shown that, even on the level of fusion rings, the numerical invariants ($\fpdim$s, central charge, etc.) associated to these categories generate a severely restricted collection of algebraic number fields.  One possibility is that the fields appearing above are a consequence of the definitions of fusion rings or categories.  This seems highly unlikely to us, and so we must believe that there exists some vast untapped source of examples whose numerical invariants will generate a more diverse collection of algebraic number fields, i.e. \!different from $\mathbb{Q}_n$ or $\mathbb{Q}(\sqrt{n})$ for $n\in\mathbb{Z}_{\geq1}$.  We present the following questions and conjectures to encourage the search for more interesting examples.

\begin{paragraph}{General discussion}We will phrase the following in terms of fusion rings, but the same constructions and open questions apply equally well, if not more so, to fusion categories.  Products of fusion rings are a cheap way to create fusion rings $R$ where $\mathbb{K}_x$ for $x\in R$ are a variety of distinct fields.  A fusion ring $R$ is \emph{prime} if it cannot be factored as $R=S\times T$ for any fusion rings $S$ and $T$.

\begin{question}
Does a prime fusion ring $R$ exist whose dimensional grading is not cyclic?
\end{question}

\begin{question}
Does a prime fusion ring $R$ exist such that $\fpdim(x),\fpdim(y)\notin\mathbb{K}_0$ for some $x,y\in R$, and $\mathbb{K}_x\cap\mathbb{K}_y=\mathbb{K}_0$?
\end{question}

The exact relationship between categorical and Frobenius-Perron dimensions is not known.  Their connection is related to the major open question which is whether or not every fusion category possesses a pivotal/spherical structure \cite[Question 4.8.3]{tcat}.

\begin{question}
Does a fusion category $\mathcal{C}$ exist with $\dim(\mathcal{C})\not\in\mathbb{K}_0?$
\end{question}

We have shown the field extension $\mathbb{K}_0/\mathbb{Q}$ is a major component of the structure of a fusion ring.  For more evidence of this, the degree of the field extension $\mathbb{K}_0/\mathbb{Q}$ for a fusion ring $R$ is bounded above by its rank.  Indeed, any fusion ring $R$ has at least $[\mathbb{K}_0:\mathbb{Q}]$ distinct \emph{formal codegrees} \cite{codegrees} consisting of $\mathrm{FPdim}(R)$ and its Galois conjugates, which are at most $\mathrm{rank}(R)$ in number, hence
\begin{equation}
[\mathbb{K}_0:\mathbb{Q}]\leq\mathrm{rank}(R).
\end{equation}
The extremity of this inequality should be of interest.
\begin{defi}
A \emph{minimal rank} fusion ring (or fusion category) $R$ has $[\mathbb{K}_0:\mathbb{Q}]=\mathrm{rank}(R)$. 
\end{defi}
Any minimal rank fusion ring is commutative \cite[Example 2.18]{ost15} as its formal codegrees are distinct, and any minimal rank fusion category is Galois conjugate to a pseudounitary fusion category.  An infinite family of examples are $\mathcal{C}(\mathfrak{sl}_2,k)_\mathrm{ad}$ where $\kappa=k+2$ is prime.

\begin{question}
Do minimal rank fusion rings (fusion categories) of each rank $n\in\mathbb{Z}_{\geq1}$ exist?
\end{question}

\begin{question}
Is every minimal rank fusion category braided? modular?  How many exist up to the appropriate equivalence?
\end{question}

\end{paragraph}

\begin{paragraph}{Central charge discussion} We know for fixed $\mathfrak{g}$, the order of $\xi(\mathcal{C}(\mathfrak{g},k))$ tends to infinity as $k\to\infty$.  Similarly, the degree of the field extension $\mathbb{K}_0/\mathbb{Q}$ tends to infinity with $k$ (Equation (\ref{quantcharge})).  Superficially this would appear to be at odds with Proposition \ref{lemprime}.  But if $N=p^b$ is a prime power for some $b\in\mathbb{Z}_{\geq1}$ and $2p^a+1$ is not prime for any $a\in\mathbb{Z}_{\geq1}$ with $1\leq a\leq b$, then $2N$ is not in the image of $\varphi$, hence there exists no $\mathcal{C}(\mathfrak{g},k)$ with $[\mathbb{K}_0:\mathbb{Q}]=N$.  Moreover, it is easy to verify that $\xi(\mathcal{C}(\mathfrak{g},k))^8=1$ if and only if $\kappa$ divides $h^\vee\dim\mathfrak{g}$, e.g. \!$\xi(\mathcal{C}(\mathfrak{sl}_2,k))^8=1$ if and only if $\kappa$ divides $6$ and therefore $k=1,4$.

\begin{theo}\label{conctheo}
Let $N=p^b$ be a prime power for some $b\in\mathbb{Z}_{\geq1}$ and assume $2p^a+1$ is not prime for any $a\in\mathbb{Z}_{\geq1}$ with $1\leq a\leq b$.  There does not exist $\mathcal{C}(\mathfrak{g},k)$ such that the exponent of $\mathrm{Gal}(\mathbb{K}_0/\mathbb{Q})$ is equal to $N$.
\end{theo}

\begin{proof}
For all $\mathcal{C}(\mathfrak{g},k)$, $\mathbb{K}_0=\mathbb{Q}_n$ for some $n\in\mathbb{Z}_{\geq1}$ (with 2 unimportant exceptions).  Hence all possible Galois groups $\mathrm{Gal}(\mathbb{Q}_n/\mathbb{Q})$ are realized by $\mathcal{C}(\mathfrak{sl}_3,k)$ for some $k\in\mathbb{Z}_{\geq1}$ for which we have $\mathbb{L}_\mathbbm{1}=\mathbb{K}_1=\mathbb{K}_0$.  But $\xi(\mathcal{C}(\mathfrak{sl}_3,k))^8=1$ if and only if $k=1,5,7,11,23$.  But Proposition \ref{lemprime} implies the exponent of $\mathrm{Gal}(\mathbb{K}_0/\mathbb{Q})$ is not $N$ because in the 5 exceptional cases of $\mathfrak{sl}_3$, $[\mathbb{K}_0:\mathbb{Q}]\leq4$.
\end{proof}
Proposition \ref{conctheo} shows that, if possible, it would require creativity to construct a modular tensor categories of the above type out of known examples.
\begin{question}
Is the bound in Corollary \ref{bettercor} tight for realizable exponents of $\mathbb{Q}_n/\mathbb{Q}$?
\end{question}
\end{paragraph}

\begin{paragraph}{Witt group discussion}  A group homomorphism $S:\mathcal{W}\to s\mathcal{W}$ was introduced in \cite[Section 5.3]{DNO} whose kernel is precisely the subgroup $\mathcal{W}_{Ising}\cong\mathbb{Z}/16\mathbb{Z}$ generated by the Witt equivalence classes of Ising braided fusion categories.  The authors then ask whether $S$ is surjective \cite[Question 5.15]{DNO}.  This question is related to the minimal modular extension conjecture for slightly degenerate braided fusion categories.  Although we do not have an answer to this question, our results allow one to truncate the question to weakly integral Witt classes where we conjecture the following.

\begin{conjecture}
The homomorphism $\left.S\right|_{\mathcal{W}_\mathbb{Q}}:\mathcal{W}_\mathbb{Q}\to s\mathcal{W}_\mathbb{Q}$ is surjective, i.e.
\begin{equation}
\mathcal{W}_\mathbb{Q}\cong s\mathcal{W}_\mathbb{Q}\times\mathcal{W}_{Ising}.
\end{equation}
\end{conjecture}

Corollary \ref{corint} ensures the image of $\left.S\right|_{\mathcal{W}_\mathbb{Q}}$ lies in $s\mathcal{W}_\mathbb{Q}$, but moreover the classes in $s\mathcal{W}_\mathbb{Q}$ are represented by integral categories (representation categories of finite-dimensonal quasi-Hopf algebras \cite[Theorem 8.33]{ENO}), which may simplify the proof somewhat.  Even less is known for larger algebraic number fields.

\begin{question}
Does an algebraic number field $\mathbb{K}\neq\mathbb{Q}$ exist such that $\mathcal{W}_\mathbb{K}=\mathcal{W}_\mathbb{Q}$?
\end{question}

\begin{question}
Are $\mathcal{W}_\mathbb{K}/\mathcal{W}_\mathbb{Q}$ finitely-generated for $\mathbb{K}\neq\mathbb{Q}$?
\end{question}
\end{paragraph}

\end{section}

\bibliography{bib}
\bibliographystyle{plain}

\end{document}